\documentclass{amsart}
\usepackage{latexsym}
\usepackage{amsfonts}
\usepackage{amssymb}
\usepackage{rotating}
\usepackage{enumerate}
\usepackage{amsmath}
\usepackage{amscd}
\usepackage{verbatim}
\usepackage{amsthm}
\usepackage{epsfig}
\usepackage{xypic}
\usepackage[bookmarks=false]{hyperref}
\def\ints{{\mathbb Z}}
\def\rats{{\mathbb Q}}

\def\reals{{\mathbb R}}
\def\complex{{\mathbb C}}

\def\proj{{\mathbb P}}
\def\aff{{\mathbb A}}
\def\FF{{\mathbb F}}

\def\im{{\text{im }}}

\def\ord{{\text{ord}}}

\DeclareMathOperator{\eff}{eff}

\def\cf{{\text{char}}}
\def\Frac{{\text{Frac}}}

\DeclareMathOperator{\Aut}{Aut}
\DeclareMathOperator{\Out}{Out}
\def\Gal{{\text{Gal}}}

\def\Spec{{\mbox{Spec }}}

\def\mc#1{\mathcal{#1}}

\def\ol#1{\overline{#1}}

\newtheorem{theorem}{Theorem}[section]
\newtheorem{prop}[theorem]{Proposition}
\newtheorem{lemma}[theorem]{Lemma}
\newtheorem{corollary}[theorem]{Corollary}
\newtheorem{predefinition}[theorem]{Definition}
\newenvironment{definition}{\begin{predefinition}\rm}{\end{predefinition}}
\newtheorem{preremark}[theorem]{Remark}
\newenvironment{remark}{\begin{preremark}\rm}{\end{preremark}}
\newtheorem{preconstruction}[theorem]{Construction}

\newtheorem{prenotation}[theorem]{Notation}

\newtheorem{preexample}[theorem]{Example}

\newtheorem{preclaim}[theorem]{Claim}

\newtheorem{prequestion}[theorem]{Question}

\numberwithin{equation}{section}
\begin{document}

\title{Fields of moduli of three-point $G$-covers with cyclic $p$-Sylow, I}

\author{Andrew Obus}
\address{Columbia University, Department of Mathematics, MC4403, 2990 Broadway, New York, NY 10027}
\email{obus@math.columbia.edu}
\thanks{The author was supported by a NDSEG Graduate Research Fellowship and an NSF Postdoctoral 
Research Fellowship in the Mathematical Sciences.  Final work on this paper was undertaken at the Max-Planck-Institut f\"{u}r Mathematik in Bonn.}

\date{\today}


\begin{abstract}
We examine in detail the stable reduction of $G$-Galois covers of the projective line over a 
complete discrete valuation field of mixed characteristic $(0, p)$, where $G$ has a \emph{cyclic} $p$-Sylow subgroup of order 
$p^n$.  If $G$ is further assumed to be \emph{$p$-solvable} 
(i.e., $G$ has no nonabelian simple composition factors with order divisible by $p$), we 
obtain the following consequence: Suppose $f: Y \to \proj^1$ is a three-point $G$-Galois cover defined over $\complex$. 
Then the $n$th higher ramification groups above $p$ 
for the upper numbering for the extension $K/\rats$ vanish, where $K$ is the field of moduli of $f$.  This extends work of 
Beckmann and Wewers.  Additionally, we completely describe the stable model of a general  three-point $\ints/p^n$-cover, 
where $p > 2$.  
\end{abstract}

\maketitle

\tableofcontents

\section{Introduction}\label{CHintro}

\subsection{Overview}\label{Soverview}
This paper focuses on understanding how primes of $\rats$ ramify in the field
of moduli of three-point Galois covers of the
Riemann sphere.  Our main result, Theorem \ref{Tmain}, generalizes results of Beckmann and Wewers (Theorems \ref{Tbeckmann} and \ref{Twewers}) about 
ramification of primes $p$ where $p$ divides the order of the Galois group and the $p$-Sylow subgroup of the Galois group is cyclic. 

Let $X$ be the Riemann sphere $\proj^1_{\complex}$, and let $f: Y \to X$ be
a finite branched cover of Riemann surfaces.  By GAGA (\cite{gaga}), $Y$ is isomorphic to an algebraic
variety, and $f$ is the analytification of an algebraic, regular map.  By a theorem of Weil, if the branch points of $f$ are
$\ol{\rats}$-rational (for example, if the cover is
branched at three points, which we can always take to be $0$, $1$, and
$\infty$---such a cover is called a
\emph{three-point cover}), then the equations of the cover $f$
can themselves be defined over $\ol{\rats}$ (in fact, over some number field).  
Let $\sigma \in \Gal(\ol{\rats}/\rats) = G_{\rats}$.  Since $X$ is defined
over $\rats$, we have that $\sigma$ acts on the set of branched covers of $X$ by acting on the coefficients of the defining
equations.  We write $f^{\sigma}: Y^{\sigma} \to X^{\sigma}$ for the cover thus obtained.  
If $f: Y \to X$ is a $G$-Galois cover, then so is $f^{\sigma}$.  Let
$\Gamma^{in} \subset G_{\rats}$ be the subgroup consisting of those $\sigma$
that preserve the isomorphism class of $f$ as
well as the $G$-action.  That is, $\Gamma^{in}$ consists of those elements $\sigma$ of $G_{\rats}$ such that there is an
isomorphism $\phi: Y \to Y^{\sigma},$ commuting with the action of $G$, that makes the following diagram commute:
\begin{equation}\label{Emoduli}
\xymatrix{
Y \ar[r]^{\phi} \ar[d]_f & Y^{\sigma} \ar[d]^{f^{\sigma}} \\
X \ar @{=}[r] & X^{\sigma}
}
\end{equation}

The fixed field $\ol{\rats}^{\Gamma^{in}}$ is known as
the \emph{field of moduli} of $f$ (as a
$G$-cover).  It is the intersection of all the fields of definition of $f$ as a
$G$-cover (i.e., those fields of
definition $K$ of $f$ such that the action of $G$ can also be written in terms
of polynomials with coefficients in $K$);
see \cite[Proposition 2.7]{CH:hu}.  

Now, since a branched $G$-Galois cover $f: Y \to X$ of the Riemann sphere is
given entirely in terms of combinatorial data 
(the branch locus $C$, the Galois group $G$, and the monodromy action of $\pi_1(X 
\backslash C)$ on $Y$),
it is reasonable to try to draw inferences about the field of moduli of $f$
based on these data.  However, not much is
known about this, and this is the goal toward which we work.  

The problem of determining the field of moduli of three-point covers has
applications toward analyzing the
\emph{fundamental exact sequence}
$$1 \to \pi_1(\proj^1_{\ol{\rats}} \ \backslash \ \{0, 1, \infty\}) \to
\pi_1(\proj^1_{\rats} \ \backslash \ \{0, 1, \infty\}) \to G_{\rats} \to 1,$$
where $\pi_1$ is the \'{e}tale fundamental group functor.  Our knowledge of this object is limited
(note that a complete understanding would yield a complete understanding of
$G_{\rats}$).  The exact sequence gives rise to an outer action of $G_{\rats}$
on $\Pi = \pi_1(\proj^1_{\ol{\rats}} \ \backslash \ \{0, 1,
\infty\})$.  This outer action would be 
particularly interesting to understand.  Knowing about fields of moduli sheds light as follows: Say the $G$-Galois
cover $f$ corresponds to the normal subgroup $N \subset \Pi$, so that $\Pi/N
\cong G$.  Then the group $\Gamma^{in}$
consists exactly of those elements of $G_{\rats}$ whose outer action on $\Pi$
both preserves $N$ and descends to an inner action on $\Pi/N \cong G$.


\subsection{Main result}

One of the first major results in this direction is due to Beckmann:

\begin{theorem}[\cite{Be:rp}]\label{Tbeckmann}
Let $f: Y \to X$ be a branched $G$-Galois cover of the Riemann sphere, with
branch points defined over $\ol{\rats}$.  
Then $p \in \rats$ can be ramified in the field of moduli of $f$ as a $G$-cover
only if $p$ is ramified in the field of
definition of a branch point, or $p \mid |G|$, or there is a collision of branch
points modulo some prime dividing $p$. In particular, if $f$ is
a three-point cover and if $p \nmid |G|$, then $p$ is unramified in the field of
moduli of $f$. 
\end{theorem}

This result was partially generalized by Wewers:

\begin{theorem}[\cite{We:br}]\label{Twewers}
Let $f: Y \to X$ be a three-point $G$-Galois cover of the Riemann sphere, and
suppose that $p$ \emph{exactly} divides
$|G|$.  Then $p$ is tamely ramified in the field of moduli of $f$ as a
$G$-cover.
\end{theorem}
In fact, Wewers shows somewhat more, in that he computes the index of tame
ramification of $p$ in the field of moduli in 
terms of some invariants of $f$.  

To state our main theorem, which is a further generalization, we will need some
group theory.  We call a finite group $G$ \emph{$p$-solvable} if its
only simple composition factors with order divisible by $p$ are isomorphic to
$\ints/p$.  Clearly, any solvable group is
$p$-solvable. 
Our main result is:

\begin{theorem}\label{Tmain}
Let $f: Y \to X$ be a three-point $G$-Galois cover of the Riemann sphere, and
suppose that a $p$-Sylow subgroup $P \subset G$ 
is cyclic of order $p^n$.  Let $K/\rats$ be the field of moduli of $f$.  
Then, if $G$ is $p$-solvable, the $n$th higher ramification groups for the upper numbering of 
(the Galois closure of) $K/\rats$ above $p$ vanish.
\end{theorem}

\begin{remark}\label{Rmainaddenda}
\begin{enumerate}[(i)]
\item Note that Beckmann's and Wewers's theorems cover the cases $n = 0, 1$ in the
notation above (and Wewers does not need the assumption of $p$-solvability).  
\item The paper \cite{Ob:fm2} will show that the result of Theorem \ref{Tmain} holds in many cases, even when $G$ is not $p$-solvable, provided 
that the normalizer of $P$ acts on $P$ via a group of order $2$.
\item If the normalizer of $P$ in $G$ is equal to the centralizer, then $G$ is always $p$-solvable.  This follows from 
\cite[Theorem 4, p.\ 169]{Za:tg}.
\item We will show (Proposition \ref{Pgroups}) that if $G$ has a cyclic $p$-Sylow subgroup and is
\emph{not} $p$-solvable, it must have a simple composition factor with order divisible by $p^n$.  There seem to be
limited examples of simple groups with cyclic $p$-Sylow subgroups
of order greater than $p$.  Furthermore, many of the examples that do exist are in the form of part (ii) of this remark (for
instance, $PSL_2(q)$, where $p^n$ exactly divides $q^2 - 1$).  
\end{enumerate}
\end{remark}

Our main technique for proving Theorem \ref{Tmain} will be an analysis of the \emph{stable reduction} of the cover $f$ to characteristic $p$ (\S\ref{Sstable}).
This is also the main technique used by Wewers in \cite{We:br} to prove Theorem \ref{Twewers}.  The argument in \cite{We:br} relies on the fact that the 
stable reduction of a three-point $G$-Galois cover to characteristic $p$ is relatively simple when $p$ exactly divides $|G|$.  When higher powers of $p$ divide $|G|$, the 
stable reduction can be significantly more complicated.  Many of the technical results needed for dealing with this situation are proven in \cite{Ob:vc}, and we will
recall them as necessary.  In particular, our proof depends on an analysis of \emph{effective ramification invariants}, which are generalizations of the invariants
$\sigma_b$ of \cite{Ra:sp} and \cite{We:br}

In proving Theorem \ref{Tmain}, we will essentially be able to reduce to the case where $G \cong \ints/p^n \rtimes \ints/m$, at the cost of having to determine the 
minimal field of definition of the \emph{stable model} of $f$, rather than just the field of moduli.  In particular, if 
the normalizer and centralizer of $P$ are equal, the proof of Theorem \ref{Tmain} boils down to understanding the stable 
model of an arbitrary three-point $\ints/p^n$-cover.  A complete description of this stable model has been given 
when $p > 3$, and in certain cases when $p=3$, in \cite{CM:sr}.  We give a complete enough description for our purposes for arbitrary $p$ (Lemma \ref{Lonenewtail}).  
Additionally, our description for $p=2$ is used in \cite{Ob:pf} to complete the proof of a product formula due to Colmez for periods of CM-abelian varieties (\cite{Co:pv}).

We should remark that, when $G \cong \ints/p^n \rtimes \ints/m$, the cover $f$ is very much like an \emph{auxiliary cover} (see \cite{Ra:sp}, \cite{We:br}, and
\cite{Ob:fm2}).  Our assumption of $p$-solvability allows us to avoid the auxiliary cover construction.

For other work on understanding stable models of mixed characteristic 
$G$-covers where the residue characteristic divides $|G|$, see for instance \cite{LM:wm}, \cite{Ma:va}, \cite{Ra:pg}, \cite{Sa:gc}, \cite{Sa:pr1}, and \cite{Sa:pr2}.  
These papers focus mostly on the case where $G$ is a $p$-group, while allowing more than three branch points.  For an application to computing the
stable reduction of modular curves, see \cite{BW:sr}.

\subsection{Section-by-section summary and walkthrough}

In \S\ref{Sgroups}--\S\ref{Sdisks}, we recall well-known facts about group theory, fields of moduli, ramification, and models of $\proj^1$.   
In \S\ref{Sreduction}, we give some explicit results on the reduction of $\ints/p^n$-torsors.  
In \S\ref{Sstable}, we recall the relevant results about stable reduction from \cite{Ra:sp} and \cite{Ob:vc}.  The most important of these is the
\emph{vanishing cycles formula}, which we then apply in the specific case of a $p$-solvable three-point cover.
In \S\ref{Sdefdata}, we recall the construction of deformation data given in \cite{Ob:vc}, which is a generalization of 
that given in \cite{He:ht}.  We also recall the \emph{effective local vanishing cycles formula} from \cite{Ob:vc}.  
In \S\ref{Squotient}, we relate the field of moduli of a cover to that of its quotient covers.
In \S\ref{Smain}, we prove our main result, Theorem \ref{Tmain}.  After reducing to a local problem, we first assume that $G \cong \ints/p^n \rtimes \ints/m$,
with $p \nmid m$.  We deal separately with the cases $m>1$ (\S\ref{Ssemidirect}) and $m = 1$ (\S\ref{Smequals1}).  Then, it is an easy application of the results
of \S\ref{Squotient} to obtain the full statement of Theorem \ref{Tmain}.

Appendix \ref{Aexplicit} gives a full description of the stable model of a general three-point $\ints/p^n$-cover, when $p > 3$
and in certain cases when $p=3$.  It uses different techniques than \cite{CM:sr}.
Furthermore, the techniques there can be adapted to give a full description whenever $p=2$ or $p=3$.   
Appendix \ref{Agroups} examines what kinds of groups with 
cyclic $p$-Sylow subgroups are not $p$-solvable, and thus are not covered by Theorem \ref{Tmain}.  
Some technical calculations from \S\ref{Sreduction} and \S\ref{Smequals1} are postponed to Appendix \ref{Asmallprimes}.
Appendices \ref{Aexplicit} and \ref{Agroups} are not necessary for the proof of Theorem \ref{Tmain}, and Appendix \ref{Asmallprimes} is only necessary when $p \leq 3$.  

\subsection{Notation and conventions}\label{Snotations}
The letter $k$ will always represent an algebraically closed field of characteristic $p>0$.

If $H$ is a subgroup of a finite group $G$, then $N_G(H)$ is the normalizer of $H$ in $G$ and $Z_G(H)$ is the centralizer of $H$ in $G$.  
If $G$ has a cyclic $p$-Sylow subgroup $P$, and $p$ is understood, we write $m_G = |N_G(P)/Z_G(P)|$.  

If $K$ is a field, then $\ol{K}$ is its algebraic closure, and $G_K$ is its absolute Galois group.  If $H \leq
G_K$, we write $\ol{K}^H$ for the fixed field of $H$ in $\ol{K}$.  Similarly, if $\Gamma$ is a group of automorphisms of a
ring $A$, we write $A^{\Gamma}$ for the fixed ring under $\Gamma$.
If $K$ is discretely valued, then $K^{ur}$ is the \emph{completion} of the maximal unramified algebraic extension of $K$.  

If $K$ is any field, and $a \in K$, then $K(\sqrt[n]{a})$ denotes a minimal 
field extension of $K$ containing an $n$th root of $a$ (\emph{not} necessarily the ring $K[x]/(x^n - a)$).  For instance, $\rats(\sqrt{9}) \cong \rats$.  
In cases where $K$ does not contain the $n$th roots of unity, it will not matter which (conjugate) extension we take.

If $R$ is any local ring, then $\hat R$ is the completion of $R$ with respect to its maximal ideal. 
If $R$ is any ring with a non-archimedean absolute value $| \cdot |$, then $R\{T\}$ 
is the ring of power series $\sum_{i=0}^{\infty} c_i T^i$ such that $\lim_{i \to \infty} |c_i| = 0$.  
If $R$ is a discrete valuation ring with fraction field $K$ of characteristic 0 and residue field of
characteristic $p$, we normalize the absolute value on $K$ and on any subring of $K$ so that $|p| = 1/p$.  
We normalize the valuation on $R$ so that $p$ has valuation $1$.

A \emph{branched cover} $f: Y \to X$ is a finite, surjective, generically \'{e}tale morphism of geometrically connected, smooth, 
proper curves. 
If $f$ is of degree $d$ and $G$ is a finite group of order $d$ with $G \cong \Aut(Y/X)$, then $f$ is called a
\emph{Galois cover with (Galois) group $G$}.  If we choose an isomorphism $i: G \to \Aut(Y/X)$, then the datum $(f, i)$
is called a \emph{$G$-Galois cover} (or just a \emph{$G$-cover}, for short).  We will usually suppress the isomorphism
$i$, and speak of $f$ as a $G$-cover.  

Suppose $f: Y \to X$ is a $G$-cover of smooth curves, and $K$ is a field of definition for $X$.  Then the \emph{field of
moduli of $f$ relative to $K$ (as a $G$-cover)} is the fixed field in $\ol{K}/K$ of $\Gamma^{in} \subset G_K$, where
$\Gamma^{in} = \{\sigma \in G_K | f^{\sigma} \cong f \text{(as } G\text{-covers)}\}$ (see \S\ref{Soverview}).  
If $X$ is $\proj^1$, then the \emph{field of moduli of $f$} means the field of moduli of $f$ relative to $\rats$.
Unless otherwise stated, a field of definition (or moduli) means a field of definition (or moduli) \emph{as a $G$-cover} (see 
\S\ref{Soverview}).  
If we do not want to consider the $G$-action, we will always explicitly say the field of definition (or moduli) \emph{as a
mere cover}.  For two covers to be isomorphic as mere covers, the isomorphism $\phi$ of \S\ref{Soverview} does not need to commute with the $G$-action.
%

\section*{Acknowledgements}
This material is mostly adapted from my PhD thesis, and I would like to thank my advisor, David Harbater, for much help 
related to this work.  I thank the referee for helping me greatly improve the exposition.  Thanks also to Bob Guralnick for inspiring many of the ideas in Appendix 
\ref{Agroups}, and to Mohamed Sa\"{i}di for useful conversations about stable reduction and vanishing cycles.

\section{Background material}\label{CHbackground}

\subsection{Finite, $p$-solvable groups with cyclic $p$-Sylow subgroups}\label{Sgroups}
The following proposition is a structure theorem on $p$-solvable groups that is integral to the paper (recall that a group $G$ is 
\emph{$p$-solvable} if its only simple composition factors with order divisible by $p$ are isomorphic to $\ints/p$).  Note that for any finite group $G$, there is a 
unique maximal prime-to-$p$ normal subgroup (as the subgroup of $G$ generated by two normal prime-to-$p$ subgroups is also normal and prime to $p$).

\begin{prop}\label{Ppsolvable}
Suppose $G$ is a $p$-solvable finite group with cyclic $p$-Sylow subgroup of order $p^n$, $n \geq 1$.
Let $N$ be the maximal prime-to-$p$ normal subgroup of $G$.  
Then $G/N \cong \ints/p^n \rtimes \ints/m_G$, where the conjugation action of
$\ints/m_G$ on $\ints/p^n$ is faithful.
\end{prop}

\begin{proof}
Clearly, $G/N$ has no nontrivial normal subgroups of prime-to-$p$ order.  Since $G$ is $p$-solvable, so is $G/N$.  Thus, a 
minimal normal subgroup of $G/N$, being the product of isomorphic simple groups (\cite[8.2, 8.3]{As:fg}), must be isomorphic to $\ints/p$.  It is readily verified that
$m_G = m_{G/N}$, so the proposition follows from 
\cite[Lemma 2.3]{Ob:vc}.
\end{proof}

\subsection{$G$-covers versus mere covers}\label{Smere}
Let $f: Y \to X$ be a $G$-cover of smooth, proper, geometrically connected curves.  Let $K$ be a field of definition for $X$, and let $L/K$ be a field containing the 
field of moduli of $f$ as a \emph{mere} cover (which is equivalent to $L$ being a field of definition of $f$ as a mere cover, see \cite[Proposition 2.5]{CH:hu}).
This gives rise to a homomorphism $h: G_L \to \Out(G)$ as follows: For $\sigma \in G_L$, consider the diagram (\ref{Emoduli}), which we reproduce here:
$$\xymatrix{
Y \ar[r]^{\phi} \ar[d]_f & Y^{\sigma} \ar[d]^{f^{\sigma}} \\
X \ar @{=}[r] & X^{\sigma}
}$$
The isomorphism $\phi$ is well-defined up to composition with an element of $G$ acting on $Y^{\sigma}$.  Thus, 
the map $h_{\sigma}$ given by $h_\sigma(g) := \phi \circ g \circ \phi^{-1}$ is well defined as an element of $\Out(G)$ 
(the input is thought of as an automorphism of $Y$, and the output is thought of as an automorphism of $Y^{\sigma}$).
Then $L$ contains the field of moduli of $f$ (as a $G$-cover) iff $h_{\sigma}$ is inner (because then there will be a choice of $\phi$ making the diagram
$G$-equivariant). 

\subsection{Wild ramification}\label{Sramification}
We state here some facts from \cite[IV]{Se:lf} and derive some consequences.
Let $K$ be a complete discrete valuation field with
residue field $k$.  If $L/K$ is a finite Galois
extension of fields with Galois group $G$, then 
$L$ is also a complete discrete valuation field with residue field $k$.  Here $G$ is
of the form $P \rtimes \ints/m$, where
$P$ is a $p$-group and $m$ is prime to $p$.  The group $G$ has a filtration $G  = G_0
\supseteq G_i$ ($i \in \reals_{\geq 0}$) for the lower numbering, and $G \supseteq G^i$ for the upper numbering 
($i \in \reals_{\geq 0}$).  If $i \leq j$, then $G_i \supseteq G_j$ and $G^i \supseteq G^j$ 
(see \cite[IV, \S1, \S3]{Se:lf}).   
The subgroup $G_i$ (resp.\ $G^i$) is known as the \emph{$i$th higher ramification group for
the lower numbering (resp.\ the upper 
numbering)}.  One knows that $G_0 = G^0 = G$, and that for sufficiently small $\epsilon > 0$, 
$G_{\epsilon} = G^{\epsilon} = P$.  For sufficiently large $i$, $G_i = G^i = \{id\}$.  
Any $i$ such that $G^i \supsetneq G^{i + \epsilon}$ for all $\epsilon > 0$ is
called an \emph{upper jump} of the extension $L/K$.  Likewise, if $G_i \supsetneq G_{i+\epsilon}$, then $i$ is called a
\emph{lower jump} of $L/K$.  The lower jumps are all prime-to-$p$ integers.
The greatest upper jump (i.e., the greatest $i$ such that $G^i \neq \{id\}$) is called the
\emph{conductor} of higher ramification of $L/K$.  The upper numbering is invariant under
quotients (\cite[IV, Proposition 14]{Se:lf}).  That is, if $H \leq G$ is normal, and $M = L^H$, then the $i$th higher
ramification group for the upper numbering for $M/K$ is $G^i/(G^i \cap H) \subseteq G/H$.

\begin{lemma}\label{Ltamenochange}
Let $K \subseteq L \subseteq L'$ be a tower of field extensions such that $L'/L$ is tame, $L/K$ is ramified, and
$L'/K, L/K$ are finite Galois.  Then the conductor of $L'/K$ is equal to the conductor of $L/K$.
\end{lemma}

\begin{proof}  
This is an easy consequence of \cite[IV, Proposition 14]{Se:lf}.
\end{proof}

\begin{lemma}\label{Lcompositum}
Let $L_1, \ldots, L_{\ell}$ be finite Galois extensions of $K$ with compositum $L$ in some algebraic closure of $K$.
Denote by $h_i$ the conductor of $L_i/K$ and by $h$ the conductor of $L/K$.  Then $h = \max_i(h_i)$.
\end{lemma}

\begin{proof}
Write $G = \Gal(L/K)$ and $N_i = \Gal(L/L_i)$.  
Suppose $g \in G^j \subseteq \Gal(L/K)$.  Since $L$ is the compositum of the $L_i$, the
intersection of the $N_i$ is trivial.  So $g$ is trivial iff its image in each $G/N_i$ is trivial.
Because the upper numbering is invariant under quotients, this shows that $G^j$ is trivial iff the $j$th higher
ramification group for the upper numbering for $L_i/K$ is trivial for all $i$.  This means that $h = \max_i(h_i)$.
\end{proof}

If $A, B$ are the valuation rings of $K, L$, respectively, sometimes we will refer to the conductor and
higher ramification groups of the extension $B/A$.  If $f:Y \to X$ is a branched cover of curves and $f(y) = x$, then we
refer to the higher ramification groups of $\hat{\mc{O}}_{Y, y}/\hat{\mc{O}}_{X, x}$ \emph{as the higher ramification groups
at $y$} (or, if $f$ is Galois, and we only care about groups up to isomorphism, as the \emph{higher ramification groups above $x$}).

We include two well-known lemmas.  The first follows easily from the Hurwitz formula (see also \cite[Propositions 3.7.8, 6.4.1]{St:af}).  For the second, 
see (\cite[Theorem 1.4.1 (i)]{Pr:fc}).

\begin{lemma}\label{Lartinschreiergenus}
Let $f:Y \to \proj^1$ be a $\ints/p$-cover of curves over an algebraically
closed field $k$ of characteristic $p$, ramified 
at exactly one point of order $p$.  If the conductor of higher ramification at this point is $h$, then the genus of $Y$ is 
$\frac{(h-1)(p-1)}{2}$.
\end{lemma}

\begin{lemma}\label{Lexplicitconductor}
Let $f: Y \to \proj^1$ be a $\ints/p$-cover of $k$-curves,
branched at one point.  Then $f$ can be given birationally by an equation
$y^p - y = g(x)$, where the terms of $g(x) \in k[x]$ have prime-to-$p$ degree (the branch point is $x = \infty$).  If $h$ is the conductor of
higher ramification at $\infty$, then $h = \deg(g)$.
\end{lemma}


\subsection{Semistable models of $\proj^1$}\label{Sdisks}
Let $R$ be a mixed characteristic $(0, p)$ complete discrete valuation ring with residue field $k$ and fraction field $K$.
If $X$ is a smooth curve over $K$, then a \emph{semistable}
model for $X$ is a relative flat curve $X_R \to \Spec R$ with $X_R \times_R K \cong X$ and semistable special fiber (i.e.,
the special fiber is reduced with only ordinary double points for singularities).  If $X_R$ is smooth, it is called a \emph{smooth model}.

\subsubsection{Models}
Let $X \cong \proj^1_K$.  Write $v$ for the valuation on $K$.
Let $X_R$ be a smooth model of $X$ over $R$.  
Then there is an element $T \in K(X)$ such that $K(T) \cong K(X)$ and the local
ring at the generic point of the special fiber of $X_R$ is the valuation ring of $K(T)$ corresponding to the
Gauss valuation (which restricts to $v$ on $K$).  We say that our model
corresponds to the Gauss valuation on $K(T)$, and we call $T$ a
\emph{coordinate} of $X_R$.  Conversely, if $T$ is
any rational function on $X$ such that $K(T) \cong K(X)$, there is a smooth
model $X_R$ of $X$ such that $T$ is a coordinate of $X_R$.
In simple terms, $T$ is a coordinate of $X_R$ iff, for all $a, b \in R$, the subvarieties of $X_R$ cut out by $T-a$ and $T-b$ intersect exactly when
$v(a-b) > 0$.  

Now, let $X_R'$ be a semistable model of $X$ over $R$.
The special fiber of $X_R'$ is a tree-like configuration of $\proj^1_k$'s.
Each irreducible component $\ol{W}$ of the special fiber $\ol{X}$ of $X_R'$ yields
a smooth model of $X$ by blowing down all other irreducible components of $\ol{X}$.  
If $T$ is a coordinate on the smooth model of $X$ with $\ol{W}$ as special fiber, we will say that $T$ corresponds to $\ol{W}$.

\subsubsection{Disks and annuli}

We give a brief overview here.  For more details, see \cite{He:dc}.

Let $X_R'$ be a semistable model for $X = \proj^1_K$.  Suppose $x$ is a smooth point
of the special fiber $\ol{X}$ of $X_R'$ on the 
irreducible component $\ol{W}$.  Let $T$ be a coordinate corresponding to $\ol{W}$ such
that $T = 0$ specializes to $x$.  Then the set of points of $X(\ol{K})$ which specialize
to $x$ is the \emph{open $p$-adic disk} $D$ given by  
$v(T) > 0$.  The ring of functions on the formal disk corresponding to $D$ is $\hat{\mc{O}}_{X, x} \cong R\{T\}$.

Now, let $x$ be an ordinary double point of $\ol{X}$, at the
intersection of components $\ol{W}$ and $\ol{W}'$.  Then the 
set of points of $X(\ol{K})$ which specialize to $x$ is an \emph{open annulus} $A$.  If $T$
is a coordinate corresponding to $\ol{W}$ such that $T=0$ specializes to
$\ol{W}'\backslash \ol{W}$, then $A$ is given by $0 <
v(T) < e$ for some $e \in v(K^{\times})$. The ring of functions on the formal annulus corresponding to $A$
is $\hat{\mc{O}}_{X, x} \cong R[[T,U]]/(TU - p^e)$.  Observe that $e$ is independent of the coordinate.  

Suppose we have a preferred coordinate $T$ on $X$ and a semistable model $X_R'$
of $X$ whose special fiber $\ol{X}$ contains an irreducible 
component $\ol{X}_0$ corresponding to the coordinate $T$.  
If $\ol{W}$ is any irreducible component of $\ol{X}$ other than $\ol{X}_0$, then
since $\ol{X}$ is a tree of $\proj^1$'s, 
there is a unique non-repeating sequence of consecutive, intersecting components $\ol{X}_0,
\ldots, \ol{W}$.  Let $\ol{W}'$ be the
component in this sequence that intersects $\ol{W}$.  Then the set of points in
$X(\ol{K})$ that specialize to the connected component of $\ol{W}$ in $\ol{X} \backslash \ol{W}'$
is a closed $p$-adic disk $D$.  If the established preferred coordinate (equivalently, the
preferred component $\ol{X}_0$) is clear, we will abuse language and refer to the component $\ol{W}$ as 
\emph{corresponding to the disk $D$}, and vice versa.  If $U$ is a coordinate corresponding to $\ol{W}$, 
and if $U = \infty$ does not specialize to the connected component of $\ol{W}$ in $\ol{X} \backslash \ol{W}'$, 
then the ring of functions on the formal disk corresponding to $D$ is $R\{U\}$.

\section{\'{E}tale reduction of torsors}\label{Sreduction}
Let $R$ be a mixed characteristic $(0, p)$ complete discrete valuation ring with residue field $k$ and fraction field $K$.
Let $\pi$ be a uniformizer of $R$.  Recall that we normalize the valuation of $p$ (not $\pi$) to be 1.
For any scheme or algebra $S$ over $R$, write $S_K$ and $S_k$ for its base
changes to $K$ and $k$, respectively.

The following lemma will be used in the proof of Lemma \ref{Lonenewtail} to analyze cyclic covers of closed $p$-adic disks given by explicit equations.

\begin{lemma}\label{Lartinschreier}
Assume that $R$ contains the $p^n$th roots of unity.  
Let $X = \Spec A$, where $A = R\{T\}$.  
Let $f: Y_K \to X_K$ be a $\mu_{p^n}$-torsor given by the equation $y^{p^n} = g$,
where $g = 1 + \sum_{i=1}^{\infty} c_iT^i$.  Suppose one of the following two conditions holds:\\

\begin{enumerate}[(i)]
\item  $\min_i v(c_i) = n + \frac{1}{p-1}$, and  
$v(c_i) > n + \frac{1}{p-1}$ for all $i$ divisible by $p$. 
\item $p$ is odd, $v(c_1) > n$, $v(c_p) > n$, and $\min_{i \ne 1, p} v(c_i) = n + \frac{1}{p-1}$.  
Also, $v(c_i) > n + \frac{1}{p-1}$ for all $i > p$ divisible by $p$.  Lastly, $v(c_p - \frac{c_1^p}{p^{(p-1)n+1}}) > n + \frac{1}{p-1}$.  
\end{enumerate}
Let $h$ be the largest $i$ $(\ne p)$ such that $v(c_i) = n + \frac{1}{p-1}$. Then $f: Y_K \to X_K$ splits into a union of $p^{n-1}$ disjoint 
$\mu_p$-torsors.  Let $Y$ be the normalization of $X$ in the total ring of fractions of $Y_K$.  Then the map $Y_k \to X_k$ is \'{e}tale, and is birationally 
equivalent to the union of $p^{n-1}$ disjoint Artin-Schreier covers of $\proj^1_k$, each with conductor $h$.
\end{lemma}

\begin{proof} 
Suppose (i) holds.  We claim that $g$ has a $p^{n-1}$st root $1 + au$ in $A$
such that $a \in R$, $v(a) = \frac{p}{p-1}$, and the reduction $\ol{u}$ of $u$ in $A_k = k[T]$ is of
degree $h$ with only prime-to-$p$ degree terms.  By \cite[Ch.\ 5, Proposition 1.6]{He:ht} (the \'{e}tale reduction case) and 
Lemma \ref{Lexplicitconductor}, this suffices to prove the lemma.
  
We prove the claim.  Write $g = 1 + bw$ with $b \in R$ and $v(b) = n + \frac{1}{p-1}$.  
Suppose $n > 1$.  Then, using the binomial theorem, a $p^{n-1}$st root of $g$ is given by 
$$\sqrt[p^{n-1}]{g} = 1 + \frac{1/p^{n-1}}{1!} bw + \frac{(1/p^{n-1})((1/p^{n-1}) - 1)}{2!} (bw)^2 +
\cdots.$$  Since $v(b) = n + \frac{1}{p-1}$, this series converges, and is in $A$.  Since the coefficients of all terms in this 
series of degree $\geq 2$ have valuation greater than  $\frac{p}{p-1}$, the series can be written as
$\sqrt[p^{n-1}]{g} = 1 + au$, where $a =\frac{b}{p^{n-1}} \in R$, 
$v(a) = \frac{p}{p-1}$, and $u$ congruent to $w$ (mod $\pi$).
By assumption, the reduction $\ol{w}$ of $w$ has degree $h$ and only prime-to-$p$ degree terms.  Thus $\ol{u}$ does as well.

Now assume (ii) holds.  It clearly suffices to show that there exists $a \in A$ such that $a^{p^n}g$ satisfies (i).
Let $a = 1 + \eta T$, where $\eta = -\frac{c_1}{p^n}$.  Now, by assumption, $v(c_1^p) - (p-1)n - 1 \geq \min(v(c_p), n + \frac{1}{p-1})$.  
Since $v(c_p) > n$, we derive that $v(c_1^p) > pn+1$.  Thus $v(\eta) > \frac{1}{p}$.  Then there exists $\epsilon \in \rats^{>0}$ such that
$$(1 + \eta T)^{p^n} \equiv 1 - c_1T - \binom{p^n}{p}\frac{c_1^p}{p^{pn}}T^p \pmod {p^{n + \frac{1}{p-1} + \epsilon}}.$$  
It is easy to show that $\binom{p^n}{p} \equiv p^{n-1} \pmod {p^n}$ for all $n \geq 1$.  Furthermore, the valuation of the
$T^i$ term ($1 \leq i \leq p^n$) in $(1 + \eta T)^{p^n}$ is greater than $\frac{i}{p} + n - v(i)$.  For any $i$ other than $1$ and $p$, this is greater than $n + \frac{1}{p-1}$
(here we use that $p$ is odd).
So $$(1 + \eta T)^{p^n} \equiv 1 - c_1T - \frac{c_1^p}{p^{(p-1)n + 1}}T^p \pmod {p^{n + \frac{1}{p-1} + \epsilon}}.$$
By the assumption that $v(c_p - \frac{c_1^p}{p^{(p-1)n+1}}) > n + \frac{1}{p-1}$, we now see that $(1+\eta T)^{p^n}g$ satisfies (i).
In particular, $(1 + \eta T)^{p^n}g \equiv 1 \pmod{p^{n+ \frac{1}{p-1}}}$, and, for any $i \ne 1, p$ such that $v(c_i) = n+\frac{1}{p-1}$, the
valuation of the coefficient of $T^i$ in $(1 + \eta T)^{p^n}g$ is $n+\frac{1}{p-1}$.
 \end{proof}

An analogous result, which is necessary to prove our main theorem in the case $p=2$, is in Appendix \ref{Asmallprimes}.

\section{Stable reduction of covers}\label{Sstable}

In \S\ref{Sstable}, $R$ is a mixed characteristic $(0, p)$ complete discrete valuation ring with residue field $k$ and fraction field $K$.
We set $X \cong \proj^1_K$, and we fix a \emph{smooth} model $X_R$ of $X$.
Let $f: Y \to X$ be a $G$-Galois cover defined over $K$, with $G$ any finite
group, such that the branch points of $f$ are defined over $K$ and their specializations 
do not collide on the special fiber of $X_R$.  Assume that $f$ is branched at at least $3$ points.
By a theorem of Deligne and Mumford (\cite[Corollary 2.7]{DM:ir}), combined with work
of Raynaud (\cite{Ra:pg}, \cite{Ra:sp}) and Liu (\cite{Li:sr}), there is a
minimal finite extension $K^{st}/K$ 
with ring of integers $R^{st}$, and a unique model $f^{st}: Y^{st} \to X^{st}$ of $f_{K^{st}} := f \times_K K^{st}$
(called the \emph{stable model} of $f$) such that

\begin{itemize}
\item The special fiber $\ol{Y}$ of $Y^{st}$ is semistable. 
\item The ramification points of $f_{K^{st}}$ specialize to
\emph{distinct} smooth points of $\ol{Y}$.
\item Any genus zero irreducible component of $\ol{Y}$ contains at least three
marked points (i.e., ramification points or points of intersection with the rest
of $\ol{Y}$).
\item $G$ acts on $Y^{st}$, and $X^{st} = Y^{st}/G$.
\end{itemize}
The field $K^{st}$ is called the minimal field of definition of the stable model of $f$. 
If we are working over a finite
extension $K'/K^{st}$ with ring of integers $R'$, we will sometimes abuse
language and call 
$f^{st} \times_{R^{st}} R'$ the stable model of $f$.  

\begin{remark}
Our definition of the stable model is the definition used in
\cite{We:br}.  This differs from the definition
in \cite{Ra:sp} in that \cite{Ra:sp} allows the ramification points to coalesce
on the special fiber.  
\end{remark}

\begin{remark}
Note that $X^{st}$ can be naturally identified with a blowup of $X \times_R R^{st}$
centered at closed points.  Furthermore, the nodes of $\ol{Y}$ lie above nodes of the special fiber $\ol{X}$
of $X^{st}$ (\cite[Lemme 6.3.5]{Ra:ab}), and $Y^{st}$ is the normalization of $X^{st}$ in $K^{st}(Y)$.
\end{remark}

If $\ol{Y}$ is smooth, the cover $f: Y \to X$ is said to have \emph{potentially
good reduction}.   If $f$ does not have
potentially good reduction, it is said to have \emph{bad reduction}.  In any case, the special fiber $\ol{f}:
\ol{Y} \to \ol{X}$ of the stable model is called the \emph{stable reduction} of $f$.  
The strict transform of the special fiber of $X_{R^{st}}$ in $\ol{X}$ is called the \emph{original
component}, and will be denoted $\ol{X}_0$.  

Each $\sigma \in G_K$ acts on $\ol{Y}$ (via its action on $Y$).  This action commutes with that of $G$ and is called the 
\emph{monodromy action}.  
Then it is known (see, for instance, \cite[Proposition 2.9]{Ob:vc}) that the extension $K^{st}/K$ is
the fixed field of the group $\Gamma^{st} \leq G_K$ consisting of those $\sigma
\in G_K$ such that $\sigma$ acts trivially on $\ol{Y}$.  Thus $K^{st}$ is clearly Galois over $K$.
Since $k$ is algebraically closed, the action of $G_K$ fixes $\ol{X}_0$ pointwise.

\begin{lemma}\label{Lgenusincluded}
Let $X_{R^{st}}$ be a smooth model for $X \times_K K^{st}$, and let $Y_{R^{st}}$ be its normalization in 
$K^{st}(Y)$.  
Suppose that the special fiber of $Y_{R^{st}}$ has irreducible components whose normalizations have genus greater than 0.  Then $X^{st}$ is a blow up of $X_{R^{st}}$
(in other words, the stable reduction $\ol{X}$ contains a component corresponding to the special fiber of
$X_{R^{st}}$).
\end{lemma}

\begin{proof}
Consider a modification $(X^{st})' \to X^{st}$, centered on the special fiber, such that $X_{R^{st}}$ is a
blow down of $(X^{st})'$.  Let $(Y^{st})'$ be the normalization of $(X^{st})'$ in $K^{st}(Y)$.
By the minimality of the stable model, we know that $X^{st}$ is obtained by blowing down components of $(X^{st})'$ such
that the components of $(Y^{st})'$ lying above them are curves of genus zero.  By our assumption, the component
corresponding to the special fiber of $X_{R^{st}}$ is not blown down in the map $(X^{st})' \to X^{st}$.  
Thus $X_{R^{st}}$ is a blow down of $X^{st}$. 
\end{proof}

\subsection{The graph of the stable reduction}
As in \cite{We:br}, we construct the (unordered) dual graph $\mc{G}$ of the stable reduction of $\ol{X}$. 
An \emph{unordered graph} $\mc{G}$ consists of a set of \emph{vertices} $V(\mc{G})$ and a set of \emph{edges} $E(\mc{G})$.  
Each edge has a \emph{source vertex} $s(e)$ and a \emph{target vertex} $t(e)$.  Each edge has an \emph{opposite
edge} $\ol{e}$, such that $s(e) = t(\ol{e})$ and $t(e) = s(\ol{e})$.  Also, $\ol{\ol{e}} = e$.

Given $f$, $\ol{f}$, $\ol{Y}$, and $\ol{X}$ as above, we construct two unordered graphs $\mc{G}$ and
$\mc{G}'$.  In our construction, $\mc{G}$ has a vertex $v$ for each irreducible
component of $\ol{X}$ and an edge $e$ for each ordered triple $(\ol{x}, \ol{W}', \ol{W}'')$, 
where $\ol{W}'$ and $\ol{W}''$ are irreducible components of $\ol{X}$ whose intersection is $\ol{x}$.  If $e$
corresponds to $(\ol{x}, \ol{W}', \ol{W}'')$, then
$s(e)$ is the vertex corresponding to $\ol{W}'$ and $t(e)$ is the vertex corresponding to
$\ol{W}''$.  The opposite edge of $e$ corresponds to $(\ol{x}, \ol{W}'', \ol{W}')$.
We denote by $\mc{G}'$ the \emph{augmented} graph of $\mc{G}$ constructed as follows: consider the set
$B_{\text{wild}}$ of branch points of $f$ with branching index divisible by $p$.  
For each $x \in B_{\text{wild}}$, we know that $x$ specializes to 
a unique irreducible component $\ol{W}_x$ of $\ol{X}$, corresponding to a vertex $A_x$ of $\mc{G}$.  
Then $V(\mc{G}')$ consists of the elements of $V(\mc{G})$ 
with an additional vertex $V_x$ for each $x \in B_{\text{wild}}$.  Also, $E(\mc{G}')$ consists of the elements of 
$E(\mc{G})$ with two additional opposite edges for each $x \in B_{\text{wild}}$, 
one with source $V_x$ and target $A_x$, and one with source $A_x$ and
target $V_x$.  We write $v_0$ for the vertex corresponding to the original component $\ol{X}_0$.  

We partially order the
vertices of $\mc{G}$ (and $\mc{G}'$) such that $v_1 \preceq v_2$ iff $v_1 = v_2$, $v_1 = v_0$, or $v_0$ and $v_2$ are in different 
connected components of $\mc{G}' \backslash v_1$.  The set of irreducible components of $\ol{X}$ inherits the partial order 
$\preceq$.  If $a \preceq b$, where $a$ and $b$ are vertices of $\mc{G}$ (or $\mc{G}'$) or irreducible components of $\ol{X}$, we say that 
$b$ lies \emph{outward} from $a$.

\subsection{Inertia groups of the stable reduction}

Maintain the notation from the beginning of \S\ref{Sstable}.

\begin{prop}[\cite{Ra:sp}, Proposition 2.4.11]\label{Pspecialram}
The inertia groups of $\ol{f}: \ol{Y} \to \ol{X}$ at points of $\ol{Y}$ are as
follows (note that points in the same
$G$-orbit have conjugate inertia groups):
\begin{enumerate}[(i)]
\item At the generic points of irreducible components, the inertia groups
are $p$-groups.  
\item At each node, the inertia group is an extension of a cyclic,
prime-to-$p$ order group, by a $p$-group generated by
the inertia groups of the generic points of the crossing components.
\item If a point $y \in Y$ above a branch point $x \in X$ specializes
to a smooth point $\ol{y}$ on a 
component $\ol{V}$ of 
$\ol{Y}$, then the inertia group at $\ol{y}$ is an extension of the
prime-to-$p$ part of the inertia group at $y$ by
the inertia group of the generic point of $\ol{V}$.
\item At all other points $q$ (automatically smooth, closed), the inertia
group is equal to the inertia group of the 
generic point of the irreducible component of $\ol{Y}$ containing $q$.
\end{enumerate}
\end{prop}
If $\ol{V}$ is an irreducible component of $\ol{Y}$, we will always write $I_{\ol{V}} \leq G$ for the inertia group of
the generic point of $\ol{V}$, and $D_{\ol{V}}$ for the decomposition group.

For the rest of this subsection, assume $G$ has a \emph{cyclic} $p$-Sylow subgroup.
When $G$ has a cyclic $p$-Sylow subgroup, the inertia groups above a generic point of an
irreducible component $\ol{W} \subset 
\ol{X}$ are conjugate cyclic groups of $p$-power order.  If they are of order
$p^i$, we call $\ol{W}$ a
\emph{$p^i$-component}.  If $i = 0$, we call $\ol{W}$ an \emph{\'{e}tale
component}, and if $i > 0$, we call $\ol{W}$ an
\emph{inseparable component}. 	 
For an inseparable component $\ol{W}$, the morphism $Y \times_X \ol{W} \to \ol{W}$ induced from $f$ corresponds to an inseparable extension 
of the function field $k(\ol{W})$.

As in \cite{Ra:sp}, we call irreducible component $\ol{W} \subseteq \ol{X}$ a \emph{tail} if it is not the original component and intersects exactly one other 
irreducible component of $\ol{X}$.
Otherwise, it is called an \emph{interior component}.  
A tail of $\ol{X}$ is called \emph{primitive} if it contains a branch
point other than the point at which it intersects the rest of $\ol{X}$.  
Otherwise it is called \emph{new}.  This follows \cite{We:br}.  
An inseparable tail that is a $p^i$-component will also be called a
\emph{$p^i$-tail}.  Thus one can speak of, for instance, ``new $p^i$-tails" or ``primitive \'{e}tale tails."

We call the stable reduction $\ol{f}$ of $f$ \emph{monotonic} 
if for every $\ol{W} \preceq \ol{W}'$, the inertia group of $\ol{W}'$ is contained in the inertia group of $\ol{W}$.  
In other words, the stable reduction is monotonic if the generic inertia does not increase as we move outward from $\ol{X}_0$ along $\ol{X}$.  

\begin{prop} \label{Pmonotonic}
If $G$ is $p$-solvable, then $\ol{f}$ is monotonic.
\end{prop}
\begin{proof} 
By Proposition \ref{Ppsolvable}, we know that there is a prime-to-$p$ group $N$ such that $G/N \cong \ints/p^n \rtimes
\ints/m$.  Since taking the quotient of a $G$-cover by a prime-to-$p$ group does not affect monotonicity, we may assume that $G \cong \ints/p^n
\rtimes \ints/m_G$.  By \cite[Remark 4.5]{Ob:vc}, it follows that $\ol{f}$ is monotonic.
\end{proof}

\begin{prop}[\cite{Ob:vc}, Proposition 2.13] \label{Pcorrectspec}
If $x \in X$ is branched of index $p^as$, where $p \nmid s$, then $x$
specializes to a $p^a$-component of $\ol{X}$.
\end{prop}

\begin{lemma}[\cite{Ra:sp}, Proposition 2.4.8]\label{Letaletail}
If $\ol{W}$ is an \'{e}tale component of $\ol{X}$, then $\ol{W}$ is a tail.
\end{lemma}

\begin{lemma}[\cite{Ob:vc}, Lemma 2.16]\label{Ltailetale}
If $\ol{W}$ is a $p^a$-tail of $\ol{X}$, 
then the component $\ol{W}'$ that intersects $\ol{W}$ is a $p^b$-component with $b > a$.
\end{lemma}

\begin{prop}\label{Pstabletail}
Suppose $f$ has monotonic stable reduction.  Let $K'/K$ be a field extension such that the following hold for each tail 
$\ol{X}_b$ of $\ol{X}$:

\begin{enumerate}[(i)] 
\item There exists a smooth point $\ol{x}_b$ of $\ol{X}$ on $\ol{X}_b$, such that $\ol{x}_b$ is fixed by $G_{K'}$.
\item There exists a smooth point
$\ol{y}_b$ of $\ol{Y}$ on some component $\ol{Y}_b$ lying above $\ol{X}_b$, such that $\ol{y}_b$ is fixed by $G_{K'}$.
\end{enumerate}
Then the stable model of $f$ can be defined over a tame extension of $K'$.
\end{prop}

\begin{proof}
We claim that $G_{K'}$ acts on $\ol{Y}$ through a group of prime-to-$p$ order.  This will yield the proposition.

Suppose $\gamma \in G_{K'}$ is such that $\gamma^p$ acts trivially on $\ol{Y}$.  For each tail $\ol{X}_b$, we have that $\gamma$ 
fixes $\ol{x}_b$.  Since $\gamma$ fixes the original component pointwise, it fixes the point of intersection of $\ol{X}_b$ with 
the rest of $\ol{X}$.  Any action on $\proj^1_k$ with order dividing $p$ and two fixed points is trivial, so $\gamma$ fixes each 
$\ol{X}_b$ pointwise.  By inward induction, $\gamma$ fixes $\ol{X}$ pointwise.  So $\gamma$ acts ``vertically" on $\ol{Y}$.  

Now, $\gamma$ also fixes each $\ol{y}_b$.  By Propositions \ref{Pspecialram} and \ref{Pcorrectspec}, the inertia of $f^{st}$ 
at $\ol{y}_b$ is an extension of a prime-to-$p$ group by the generic inertia of 
$f^{st}$ on $\ol{Y}_b$.  So some prime-to-$p$ power $\gamma^i$ of $\gamma$ fixes $\ol{Y}_b$ pointwise.  Since $p \nmid i$ and the
action of $\gamma$ has order $p$, it follows that $\gamma$ fixes $\ol{Y}_b$ pointwise.  
Since $\gamma$ commutes with $G$, then $\gamma$ fixes all components above $\ol{X}_b$ 
pointwise.

We proceed to show that $\gamma$ acts trivially on $\ol{Y}$ by inward induction.  
Suppose $\ol{W}$ is a component of $\ol{X}$ such that, if $\ol{W}' \succ \ol{W}$, then $\gamma$ 
fixes all components above $\ol{W}'$ pointwise.  Suppose $\ol{W}' \succ \ol{W}$ is a component such that $\ol{W}' \cap \ol{W} = 
\{\ol{w}\} \neq \emptyset$.  Let $\ol{V}$ be a component of $\ol{Y}$ above $\ol{W}$, and let $\ol{v}$ be a point of $\ol{V}$ above 
$\ol{w}$.  By the inductive hypothesis, $\gamma$ fixes $\ol{v}$.  Since $\ol{f}$ is monotonic, Proposition \ref{Pspecialram} shows 
that the $p$-part of the inertia group at $\ol{v}$ is the same as the generic inertia group of $\ol{V}$.  Thus $\gamma$ fixes $\ol{V}$ pointwise.  
Because $\gamma$ commutes with $G$, it fixes all components above $\ol{W}$ pointwise.  This completes the induction.
\end{proof}

\subsection{Ramification invariants and the vanishing cycles formula}\label{Svancycles}

Maintain the notation from the beginning of \S\ref{Sstable}, and assume additionally that $G$ has a \emph{cyclic} $p$-Sylow group $P$.
Recall that $m_G = |N_G(P)/Z_G(P)|$.  Below, we define the \emph{effective ramification invariant} $\sigma_b$ corresponding to each tail 
$\ol{X}_b$ of $\ol{X}$.

\begin{definition}\label{Draminvariant}
Consider a tail $\ol{X}_b$ of $\ol{X}$.  Suppose $\ol{X}_b$ intersects the rest of $\ol{X}$ at $x_b$.
Let $\ol{Y}_b$ be a component of $\ol{Y}$ lying above $\ol{X}_b$, and let $y_b$ be a point
lying above $x_b$.  Then the \emph{effective ramification invariant} $\sigma_b$ is defined as follows:
If $\ol{X}_b$ is an \'{e}tale tail, then $\sigma_b$ is the conductor of higher ramification
for the extension $\hat{\mc{O}}_{\ol{Y}_b, y_b}/\hat{\mc{O}}_{\ol{X}_b, x_b}$ (see \S\ref{Sramification}).  If
$\ol{X}_b$ is a $p^i$-tail ($i > 0$), then the extension $\hat{\mc{O}}_{\ol{Y}_b, y_b}/\hat{\mc{O}}_{\ol{X}_b, x_b}$ can be 
factored as $\hat{\mc{O}}_{\ol{X}_b, x_b} \stackrel{\alpha}{\hookrightarrow} S \stackrel{\beta}{\hookrightarrow} 
\hat{\mc{O}}_{\ol{Y}_b, y_b}$, where $\alpha$ is Galois and $\beta$ is purely inseparable of degree $p^i$.
Then $\sigma_b$ is the conductor of higher ramification for the extension $S/\hat{\mc{O}}_{\ol{X}_b, x_b}$.
\end{definition}

The vanishing cycles formula (\cite[3.4.2 (5)]{Ra:sp}) is a key formula that helps us understand the
structure of the stable reduction of a branched $G$-cover of curves in the case where $p$ exactly divides the order of $G$. 
The following theorem, which is the most important ingredient in the proof of Theorem \ref{Tmain}, generalizes the vanishing cycles formula to the case where $G$ 
has a cyclic $p$-Sylow group of arbitrary order.  

\begin{theorem}[Vanishing cycles formula, \cite{Ob:vc}, Corollary 3.15]\label{Tvancycles}
Let $f: Y \to X \cong \proj^1$ be a $G$-Galois cover over $K$ with bad reduction, branched only above $\{0, 1, \infty\}$, where $G$ has a cyclic
$p$-Sylow subgroup.  Let $\ol{f}: \ol{Y} \to
\ol{X}$ is the stable reduction of $f$.  Let $B_{\text{new}}$ be an indexing set for the new \'{e}tale tails and let
$B_{\text{prim}}$ be an indexing set for the primitive \'{e}tale tails.  
Then we have the formula 

\begin{equation}\label{Evancycles} 
\sum_{b \in B_{\text{new}}} (\sigma_b - 1) + \sum_{b \in B_{\text{prim}}} \sigma_b = 1.
\end{equation}
\end{theorem}

\begin{lemma}[\cite{Ob:vc}, Proposition 4.1]\label{Linsepint}
If $b$ indexes an inseparable tail $\ol{X}_b$, then $\sigma_b$ is an integer. 
\end{lemma}

\begin{lemma} [\cite{Ob:vc}, Lemma 4.2 (i)]\label{Ltailbounds}
A new tail $\ol{X}_b$ (\'{e}tale or inseparable) has $\sigma_b \geq 1 + 1/m$.  
\end{lemma}

\begin{lemma}\label{Linsep2etale}
Suppose $\ol{X}_b$ is a new inseparable $p^i$-tail with effective ramification invariant $\sigma_b$.  
Suppose further that the
inertia group $I \cong \ints/p^i$ of some component $\ol{Y}_b$ above $\ol{X}_b$ is normal in $G$.  Then, $\ol{X}_b$
is a new (\'{e}tale) tail of the stable reduction of the quotient cover $f': Y/I \to X$, with effective ramification
invariant $\sigma_b$.
\end{lemma}

\begin{proof}
Let $(f')^{st}: (Y')^{st} \to (X')^{st}$ be the stable model of $f'$.  Then, since $(Y^{st})/I$ is a semistable
model of $Y/I$, we have that $(Y')^{st}$ is a contraction of $(Y^{st})/I$.  Thus $(X')^{st}$ is a contraction of
$X^{st}$.  To prove the lemma, it suffices to prove that $\ol{X}_b$ is not contracted in the map 
$\alpha: X^{st} \to (X')^{st}$.  

By Lemmas \ref{Linsepint} and  \ref{Ltailbounds}, we know $\sigma_b \geq 2$.  A calculation using the Hurwitz formula (cf.\
\cite[Lemme 1.1.6]{Ra:sp}) shows that the genus of $\ol{Y}_b$ is greater than zero.  Since the quotient morphism $Y \to Y/I$ is
radicial on $\ol{Y}_b$, the normalization of $X^{st}$ in $K^{st}(Y/I)$ has irreducible components of genus greater than zero
lying above $\ol{X}_b$.  By Lemma \ref{Lgenusincluded}, $\ol{X}_b$ is a component of the special fiber of $(X^{st})'$,
thus it is not contracted by $\alpha$.
\end{proof}

\begin{prop}\label{Pnoinsep}
Let $f: Y \to X = \proj^1_K$ be a three-point $G$-cover with bad reduction, where $G$ is $p$-solvable, $G$ has cyclic $p$-Sylow subgroup, and $m_G > 1$.
Then $\ol{X}$ has no inseparable tails or new tails.
\end{prop}
\begin{proof} 
Since taking the quotient of a $G$-cover by a prime-to-$p$ group affects neither ramification invariants (Lemma \ref{Ltamenochange})
nor inseparability, we may assume by Proposition \ref{Ppsolvable} that $G \cong \ints/p^n \rtimes \ints/m_G$.   Then all elements of 
$G$ have either $p$-power order or prime-to-$p$ order.  The resulting cover is
branched at three points (otherwise it would be cyclic), and at least two of these points have prime-to-$p$ branching index.

We first show there are no inseparable tails.  Say there is an inseparable $p^i$-tail $\ol{X}_b$ with
effective ramification invariant $\sigma_b$.  By Lemma \ref{Linsepint}, $\sigma_b$ is an integer.  
By Lemma \ref{Ltailbounds}, $\sigma_b > 1$ if $\ol{X}_b$ does not contain the
specialization of any branch point.  Assume for the moment that this is the case. 
Then $\sigma_b \geq 2$.  Let $I$ be the common inertia group of all components of $\ol{Y}$ above $\ol{X}_b$.  If $f': Y/I \to X$ is the quotient cover, then
we know $f'$ is branched at three points, with at least two having prime-to-$p$ ramification index.  Thus the stable reduction $\ol{f}'$ has at least two primitive tails.  
By Lemma \ref{Linsep2etale}, it also has a new tail corresponding to the image of $\ol{X}_b$, which has
effective ramification invariant $\sigma_b \geq 2$.
Then the left-hand side of (\ref{Evancycles}) for the cover $f'$ is greater than
1, so we have a contradiction.

We now prove that no branch point of $f$ specializes to $\ol{X}_b$.  
By Proposition \ref{Pcorrectspec}, such a branch point $x$ would have ramification index $p^is$, where $p \nmid s$.   
Since $i \geq 1$, the only possible branching index for $x$ is $p^i$ (as it must be the order of an element of $G$).
So in $f': Y/I \to X$, $x$ has ramification index 1.  
Thus $Z \to X$ is branched in at most two points, which contradicts the
fact that $f'$ is not cyclic.

Now we show there are no new tails.  Suppose there is a new tail $\ol{X}_b$ with ramification
invariant $\sigma_b$.  If $\sigma_b \in \ints$, we get the same contradiction as in the inseparable
case.  If $\sigma_b \notin \ints$, and if $\ol{Y}_b \subseteq \ol{Y}$ is an irreducible component above $\ol{X}_b$, then $\ol{Y}_b \to \ol{X}_b$ is a
$\ints/p^i \rtimes \ints/m_b$-cover branched at only one point, where $i \geq 1$ and $m_b > 1$.
This violates the easy direction of Abhyankar's Conjecture, as this group is not quasi-$p$ (see, for instance, \cite[XIII, Corollaire 2.12]{sga1}).
\end{proof}

\section{Deformation data}\label{Sdefdata}
Deformation data arise naturally from the stable reduction of covers.  Much information is lost when we pass from the stable model of a cover to its stable reduction, 
and deformation data provide a way to retain some of this information.  This process is described in detail in
\cite[\S3.2]{Ob:vc}, and we recall some facts here.

\subsection{Generalities}\label{Sgeneralities}
Let $\ol{W}$ be any connected smooth proper curve over $k$.  
Let $H$ be a finite group and $\chi$ a 1-dimensional character 
$H \to \FF_p^{\times}.$  A \emph{deformation datum} over
$\ol{W}$ of type $(H, \chi)$ is an ordered pair $(\ol{V}, \omega)$ such that: $\ol{V} \to \ol{W}$ is
an $H$-cover;
$\omega$ is a meromorphic differential form on $\ol{V}$ that is either logarithmic or
exact (i.e., $\omega = du/u$ or $du$ for 
$u \in k(\ol{V})$); and $\eta^*\omega = \chi(\eta)\omega$ for all $\eta \in H$.  If $\omega$
is logarithmic (resp.\ exact), the deformation datum is called
multiplicative (resp.\ additive).  When $\ol{V}$ is understood, we will sometimes
speak of the deformation datum $\omega$.  

If $(\ol{V}, \omega)$ is a deformation datum, and $w \in \ol{W}$ is a closed point, we
define $m_w$ to be the order of the 
prime-to-$p$ part of the ramification index of $\ol{V} \to \ol{W}$ at $w$.  Define $h_w$
to be $\ord_v(\omega) + 1$, where $v \in
\ol{V}$ is any point which maps to $w \in \ol{W}$.  This is well-defined because $\eta^*\omega$ is a nonzero scalar multiple of $\omega$ for
$\eta \in H$.

Lastly, define $\sigma_x = h_w/m_w$.  We call $w$ a \emph{critical point} of the
deformation datum $(\ol{V}, \omega)$ if 
$(h_w, m_w) \ne (1, 1)$.  Note that every deformation datum contains only a
finite number of critical points.  The
ordered pair $(h_w, m_w)$ is called the \emph{signature} of $(\ol{V}, \omega)$ (or of
$\omega$, if $\ol{V}$ is understood) at $w$, and
$\sigma_w$ is called the \emph{invariant} of the deformation datum at $w$.

\subsection{Deformation data arising from stable reduction.}\label{Sdefdatastable}
Maintain the notation of \S\ref{Sstable}.  In particular, $X \cong \proj^1_K$, we have a $G$-cover $f:Y \to X$ defined over $K$ with bad reduction and at least three 
branch points, there is a smooth model of $X$ where the reductions of the branch points do not coalesce, and
$f$ has stable model $f^{st}: Y^{st} \to X^{st}$ and stable reduction $f: \ol{Y} \to \ol{X}$.  We assume further that 
$G$ has a cyclic $p$-Sylow subgroup.  For each irreducible component of $\ol{Y}$ lying above a
$p^r$-component of $\ol{X}$ with $r > 0$, we obtain $r$ different deformation data.  The details of this construction are given in 
\cite[Construction 3.4]{Ob:vc}, and we only give a sketch here.

Suppose $\ol{V}$ is an irreducible component of $\ol{Y}$ with generic point $\eta$
and nontrivial generic inertia group $I \cong \ints/p^r \subset G$.  We
write $B = \hat{\mc{O}}_{Y^{st}, \eta}$, and $C = B^I$.  The map $\Spec B \to \Spec C$ is given by a tower of $r$ 
maps, each of degree $p$.  We can write these maps as $\Spec C_{i+1} \to \Spec C_i$, for $1 \leq i \leq r$, such that $B = 
C_{r+1}$ and $C = C_1$.  
After a possible finite extension $K'/K^{st}$, each of these maps is given by an equation $y^p = z$ on the generic fiber, where 
$z$ is well-defined up to raising to a prime-to-$p$ power.  The morphism on the special fiber is purely inseparable.
To such a degree $p$ map, \cite[Ch.\ 5, D\'{e}finition 1.9]{He:ht} associates a meromorphic differential form $\omega_i$, well defined up 
to multiplication by a scalar in $\FF_p^{\times}$, on the special 
fiber $\Spec C_i \times_{R^{st}} k = \Spec C_i/\pi$, where $\pi$ is a uniformizer of $R^{st}$.  
This differential form is either logarithmic or exact.  Since $C/\pi \cong k(\ol{V})^{p^r} 
\cong k(\ol{V})^{p^{r-i+1}} \cong C_i/\pi$ for any $i$, each $\omega_i$ can be thought of as a differential form on $\ol{V}' = 
\Spec C \times_{R^{st}} k$, where $k(\ol{V}') = k(\ol{V})^{p^r}$.  

Let $H = D_{\ol{V}}/I_{\ol{V}} \cong D_{\ol{V}'}$.  If $\ol{W}$ is the component of $\ol{X}$ lying below $\ol{V}$, we have that 
$\ol{W} = \ol{V}'/H$.  In fact, each $(\ol{V}', \omega_i)$, for $1 \leq i \leq r$, is a deformation datum of type $(H, \chi)$ over 
$\ol{W}$, where $\chi$ is given by the conjugation action of $H$ on $I_{\ol{V}}$.  The invariant of $\sigma_i$ at a point $w \in 
W$ will be denoted $\sigma_{i,w}$.  
We will sometimes call the deformation datum $(\ol{V}', \omega_1)$ the \emph{bottom 
deformation datum} for $\ol{V}$.  

For $1 \leq i \leq r$, denote the valuation of the different of $C_i \hookrightarrow C_{i+1}$ by $\delta_{\omega_i}$.  
If $\omega_i$ is multiplicative, then $\delta_{\omega_i} = 1$.  
Otherwise, $0 < \delta_{\omega_i} < 1$.

For the rest of this section, we will only concern ourselves with deformation
data that arise from stable reduction in the manner described above.  We will use the notations of \S\ref{Sstable} 
throughout.

\begin{lemma}[\cite{Ob:vc}, Lemma 3.5, cf.\ \cite{We:br}, Proposition 1.7]\label{Lcritical}
Say $(\ol{V}', \omega)$ is a deformation datum arising from the stable reduction
of a cover, and let $\ol{W}$ be the component of $\ol{X}$ lying under
$\ol{V}'$.  Then a critical point $x$ of the
deformation datum on $\ol{W}$ is either a singular point of $\ol{X}$ or the
specialization of a branch point of $Y \to
X$ with ramification index divisible by $p$.  In the first case, $\sigma_x \ne
0$, and in the second case, $\sigma_x = 0$
and $\omega$ is logarithmic.
\end{lemma}

\begin{prop}\label{Pmultdefdata}
Let $(\ol{V}', \omega_1)$ be the bottom deformation datum for some irreducible
component $\ol{V}$ of $\ol{Y}$.  If $\omega_1$ is multiplicative, then $\omega_i = \omega_1$ for $2 \leq i \leq r$.  In particular, all $\omega_i$
are multiplicative.
\end{prop}

\begin{proof}
As is mentioned at the beginning of \cite[\S3.2.2]{Ob:vc}, we may work over a finite extension $K'/K^{st}$ containing the $p^r$th roots of unity.
Let $B$ and $C$ be as in our construction of deformation data.  Let $R'$ be the ring of integers of $K'$.
By Kummer theory, we can write $B \otimes_{R'} K' = (C \otimes_{R'}
K')[\theta]/(\theta^{p^r} - \theta_1)$.  After a further extension of $K'$, we can assume $v(\theta_1) = 0$.

By \cite[Ch.\ 5, D\'{e}finition 1.9]{He:ht},  
if $\omega_1$ is logarithmic, then the reduction $\ol{\theta}_1$ of $\theta_1$
to $k$ is not a $p$th power in $C \otimes_{R'} k$.  Again, by \cite[Ch.\ 5, D\'{e}finition 1.9]{He:ht}, 
we thus know that $\omega_1 = d\ol{\theta}_1/\ol{\theta}_1$.  It easy to
see that $\omega_i$ arises from the equation $y^p = \theta_i$ where $\theta_i =
\sqrt[p^{i-1}]{\theta_1}$.  Under the $p^{i-1}$st power isomorphism $\iota: C_i \otimes_{R'} k \to C
\otimes_{R'} k$, $\iota(\theta_i) = \theta_1$.  So, again by \cite[Ch.\ 5, D\'{e}finition 1.9]{He:ht}, $\omega_i$ is logarithmic, 
and is equal to $\frac{d\theta_1}{\theta_1}$, which is equal to $\omega_1$.
\end{proof}

\begin{lemma}\label{Lmultdd}
If $f$ is a three-point cover, then the original component of $\ol{X}$ is a $p^n$-component, and all deformation data above the original component are multiplicative.
\end{lemma}
\begin{proof} 
Since $G$ is $p$-solvable, we know by Proposition \ref{Ppsolvable} that $f:Y \to
X$ has a quotient cover $f': Y' \to X$ with Galois group $\ints/p^n \rtimes \ints/m_G$.  Since $Y \to Y'$ is of prime-to-$p$ degree, we may assume that $Y = Y'$ and $G 
\cong \ints/p^n \rtimes \ints/m_G$.
Let $J < G$ be the unique subgroup of order $p^{n-1}$.
Then the quotient cover
$\eta: Z = Y/J \to X$ has Galois group $\ints/p \rtimes \ints/m_G$.  
If all branch points of $\eta$ have prime-to-$p$ branching index, then
\cite[\S1.4]{We:mc} shows that $\eta$ is of \emph{multiplicative type} in the
language of \cite{We:mc}. 
Then $\eta$ has bad reduction by \cite[Corollary 1.5]{We:mc},
and the original component for the stable reduction $\ol{Z} \to \ol{X}$ is a
$p$-component. 
Furthermore, the deformation datum on the irreducible component of $\ol{Z}$
above the original component of $\ol{X}$ is 
multiplicative (also due to \cite[Corollary 1.5]{We:mc}).  

If $\eta$ has a branch point $x$ with ramification index divisible by $p$, then
$\eta$ has bad reduction.  By Proposition
\ref{Pcorrectspec}, $x$ specializes to a $p$-component.  
By \cite[Theorem 2, p.\ 992]{We:br}, this is the original component $\ol{X}_0$, which is the only $p$-component.
The deformation datum above $\ol{X}_0$ must be multiplicative here, as $\ol{X}_0$ contains the specialization of a branch point 
with $p$ dividing the branching index (see Lemma \ref{Lcritical}).

So in all cases, the original component is a $p$-component for $\eta$ with
multiplicative deformation datum.
Thus the bottom deformation datum above $\ol{X}_0$ for $f$ is
multiplicative.  Now, we claim that $\ol{X}_0$ is a $p^n$-component for $f$.  Let $I$ be the inertia group of a component of $\ol{Y}$ lying above $\ol{X}_0$.  
Since $\eta$ is inseparable above $\ol{X}_0$, we must have that $I \supsetneq J$.  
Thus $|I| = p^n$, proving the claim.  Finally, Proposition \ref{Pmultdefdata} shows that all the deformation data
above $\ol{X}_0$ for $f$ are multiplicative.
\end{proof}

\subsection{Effective invariants of deformation data}\label{Seffinvs}
Maintain the notations of \S\ref{Sdefdatastable}
Recall that $\mc{G}'$ is the augmented dual graph of $\ol{X}$.  To each edge $e$ of $\mc{G}'$ we will associate an
invariant $\sigma^{\eff}_e$, called the \emph{effective invariant}.  

\begin{predefinition}[cf.\ \cite{Ob:vc}, Definition 3.10]\label{Dsigmaeff} 
\rm \ \vspace{0.02cm}
\begin{itemize}
\item If $s(e)$ corresponds to a $p^r$-component $\ol{W}$ and $t(e)$ corresponds to a $p^{r'}$-component $\ol{W}'$ with 
$r \geq r'$, then $r \geq 1$ by Lemma \ref{Letaletail}. 
Let $\omega_i$, $1 \leq i \leq r$, be the deformation data above $\ol{W}$.  If $\{w\} = \ol{W} \cap \ol{W}'$, define 
$\sigma_{i,w}$ to be the invariant of $\omega_i$ at $w$.  Then 
$$\sigma^{\eff}_e := \left( \sum_{i=1}^{r-1} \frac{p-1}{p^{i}}\sigma_{i,w} \right) + \frac{1}{p^{r-1}}\sigma_{r,w}.$$
Note that this is a weighted average of the $\sigma_{i,w}$'s.
\item If $s(e)$ corresponds to a $p^r$-component and $t(e)$ corresponds to a $p^{r'}$-component with $r < r'$, then
$\sigma^{\eff}_e := -\sigma^{\eff}_{\ol{e}}$.  
\item If either $s(e)$ or $t(e)$ is a vertex of $\mc{G}'$ but not $\mc{G}$, then $\sigma^{\eff}_e := 0$.  
\end{itemize}
\rm
\end{predefinition}

\begin{lemma}[\cite{Ob:vc}, Lemma 3.11 (i), (iii)]\label{Lsigmaeffcompatibility}
\begin{enumerate}[(i)]
\item For any $e \in E(\mc{G}')$, we have $\sigma^{\eff}_e = -\sigma^{\eff}_{\ol{e}}$.
\item  If $t(e)$ corresponds to an \'{e}tale tail $\ol{X}_b$, then $\sigma^{\eff}_e = \sigma_b$.
\end{enumerate}
\end{lemma}

\begin{lemma}[Effective local vanishing cycles formula, \cite{Ob:vc}, Lemma 3.12]\label{Lgenlocvancycles}
Let $v \in V(\mc{G}')$ correspond to a $p^j$-component $\ol{W}$ of $\ol{X}$ with genus $g_v$.  Then
$$\sum_{s(e) = v} (\sigma^{\eff}_e - 1) = 2g_v - 2.$$
\end{lemma}

\begin{lemma}\label{Lsigmaeff}
Let $e$ be an edge of $\mc{G}$ such that $s(e) \prec t(e)$.  Write $\ol{W}$ for the component corresponding to $t(e)$.
Let $\Pi_e$ be the set of branch points of $f$ with branching index divisible by $p$ that specialize to or outward from $\ol{W}$.  
Let $B_e$ index the set of \'{e}tale tails $\ol{X}_b$ such that $\ol{X}_b \succeq \ol{W}$.  Then the following formula holds:
$$\sigma^{\eff}_e - 1 = \sum_{b \in B_e} (\sigma_b - 1) - |\Pi_e|.$$
\end{lemma}

\begin{proof}
For the context of this proof, call a set $A$ of edges of $\mc{G}'$ \emph{admissible} if
\begin{itemize}
\item For each $a \in A$, we have $s(e) \preceq s(a) \prec t(a)$.
\item For each $b \in B_e$, there is exactly one $a \in A$ such that $t(a) \preceq v_b$, where $v_b$ is the vertex 
corresponding to $\ol{X}_b$.
\item For each $c \in \Pi_e$, there is exactly one $a \in A$ such that $t(a) \preceq v_c$, where $v_c$ is the vertex 
corresponding to $c$.
\end{itemize}

For an admissible set $A$, write $F(A) = \sum_{a \in A} (\sigma^{\eff}_a - 1)$.  We claim that 
$F(A) = \sum_{b \in B_e} (\sigma_b - 1) - |\Pi_e|$ for all admissible $A$.  
Since the set $\{e\}$ is clearly admissible, this claim proves the lemma. 

Now, if $A$ is an admissible set of edges, then we can form a new admissible set $A'$ by eliminating an edge $\alpha$ such 
that $t(\alpha)$ is not a leaf of $\mc{G}'$, and replacing it with the set of all edges $\beta$ such that $t(\alpha) = s(\beta)$.  
Since $t(\alpha)$ always corresponds to a vertex of genus $0$, Lemmas \ref{Lsigmaeffcompatibility} (i) and \ref{Lgenlocvancycles}
show that $F(A) = F(A')$.  By repeating this process, we see that $F(A) = F(D)$, where 
$D$ consists of all edges $d$ such that $t(d) = v_b$ or $t(d)= v_c$, with $b \in B_e$ or $c \in \Pi_e$.  But by Lemma
\ref{Lsigmaeffcompatibility} (ii), $F(D) = \sum_{b \in B_e} (\sigma_b - 1) + \sum_{c \in |\Pi_e|} (0 - 1)$, proving the
claim.
\end{proof}

The remainder of this section will be used only in Appendix \ref{Aexplicit}, and may be skipped by a reader who does not wish to 
read that section.

Consider two intersecting components $\ol{W}$ and $\ol{W}'$ of $\ol{X}$ as in Definition \ref{Dsigmaeff}.  Suppose $\ol{W}$ is a $p^r$-component and 
$\ol{W}'$ is a $p^{r'}$-component, $r \geq r$.
If $\ol{V}$ and $\ol{V}'$ are intersecting components lying above $\ol{W}$ and $\ol{W}'$, respectively, then
for each $i$, $1 \leq i \leq r$, there is a deformation datum with
differential form $\omega_i$ associated to $\ol{V}$.  Likewise, for each $i'$, $1 \leq i' \leq r'$, there is a deformation datum
with differential form $\omega'_{i'}$ associated to $\ol{V}'$.  Let $(h_{i, w}, m_w)$ be the invariants of $\omega_i$ at 
$w$, the intersection point of $\ol{W}$ and $\ol{W}'$.  Suppose $v$ is an intersection point of $\ol{V}$ and 
$\ol{V}'$.  We have the following proposition relating the change in the
differents of the deformation data (see just before Lemma \ref{Lcritical}) and the \'{e}paisseur of the annulus corresponding to $w$:

\begin{prop}\label{Pdifferentepaisseur}
Let $\epsilon_w$ be the \'{e}paisseur of the formal annulus corresponding to $w$.  
\begin{itemize}
\item If $i = i' + r - r'$, then $\delta_{\omega_i} - \delta'_{\omega'_{i'}} =
\frac{\epsilon_w\sigma_{i,w}(p-1)}{p^{i}}$.
\item If $i \leq r-r'$, then $\delta_{\omega_i} =
\frac{\epsilon_w\sigma_{i,w}(p-1)}{p^{i}}$.
\end{itemize}
\end{prop}

\begin{proof}
Write $I_i$ for the unique subgroup of order $p^i$ of the inertia group of $\ol{f}$ at $v$ in $G$.
Let $\mc{A} = \Spec \hat{\mc{O}}_{Y^{st}, v}$.
Let $\epsilon$ be the \'{e}paisseur of $\mc{A}/(I_{r - i+1})$.  Then, in the case $i = i' + r - r'$,
\cite[Ch.\ 5, Proposition 1.10]{He:ht} shows that
$\delta_{\omega_i} - \delta'_{\omega'_{i'}} = \epsilon h_{i,w}(p-1)$.  
In the case $i < r-r'$, the same proposition shows $\delta_{\omega_i} - 0 = \epsilon h_{i,w}(p-1)$ 
Also, \cite[Proposition 2.3.2 (a)]{Ra:sp} shows that $\epsilon_w = p^{i}m_w\epsilon$.  The proposition follows.
\end{proof}

It will be useful to work with the \emph{effective different}, which we define now.
\begin{definition}\label{Deffdifferent}
Let $\ol{W}$ be a $p^r$-component of $\ol{X}$, and let $\omega_i$, $1 \leq i \leq r$, be the deformation data above 
$\ol{W}$.  Define the \emph{effective different} $\delta^{\eff}_{\ol{W}}$ by
$$\delta^{\eff}_{\ol{W}} = \left( \sum_{i=1}^{r-1} \delta_{\omega_i} \right) + \frac{p}{p-1}\delta_{\omega_r}.$$
\end{definition}

\begin{lemma}\label{Leffdifferentepaisseur}
Assume the notations of Proposition \ref{Pdifferentepaisseur}.  Let $e$ be an edge of $\mc{G}$ such that $s(e)$ corresponds to 
$\ol{W}$ and $t(e)$ corresponds to $\ol{W}'$.  Then
$$\delta^{\eff}_{\ol{W}} - \delta^{\eff}_{\ol{W}'} = \sigma^{\eff}_e\epsilon_w.$$
\end{lemma}

\begin{proof}
We sum the equations from Proposition \ref{Pdifferentepaisseur} for $1 \leq i \leq r-1$.  
Then we add $\frac{p}{p-1}$ times the equation for $i = r$.  This exactly gives
$\delta^{\eff}_{\ol{W}} - \delta^{\eff}_{\ol{W}'} = \sigma^{\eff}_e\epsilon_w.$
\end{proof}

\section{Quotient covers}\label{Squotient}

In this section, we relate the minimal field of definition of the stable model of a $G$-cover to that of its quotient $G/N$-covers, when $p \nmid |N|$.  
This allows a significant simplification of the group theory in \S\ref{Smain}.
 
\begin{lemma}\label{Lstr2aux2}
Let $f: Y \to X$ be any $G$-Galois cover of smooth, proper, geometrically connected curves over any field 
(we do not assume that a $p$-Sylow subgroup of $G$ is cyclic).  Suppose $G$ has a normal subgroup $N$ such that 
$p \nmid |N|$, and $Z := Y/N$.  So $f$ factorizes as
$$Y \stackrel{q}{\to} Z \stackrel{\eta}{\to} X.$$  Suppose $L$ is a field such that $\eta: Z \to X$ is defined over 
$L$, and let $Z_L$ be a model for $Z$ over $L$.  Suppose further that $q:Y \to Z$ can be defined over
$L$, with respect to the model $Z_L$.  Then the field of moduli $L'$ of $f$ with respect to $L$ satisfies $p \nmid [L':L]$.
\end{lemma}

\begin{proof}
Clearly, $f$ is defined as a \emph{mere} cover over $L$.  So let $Y_L$ be a model for $Y$ over $L$ such that $Y_L/N = Z_L$ (and set
$X_L = Z_L/(G/N)$).
Then the cover $Y_L \to X_L$ gives rise to a homomorphism $h: G_L \to \Out(G)$ as in \S\ref{Smere}, 
whose kernel is the subgroup of $G_L$ fixing the field of moduli of $f$.  
Since $q$ is defined over $L$, the image of $h$ acts by inner automorphisms on $N$.  Thus, there is a natural homomorphism $r: (\im h) \to \Out(G/N)$.  
Since $\eta$ is defined over $L$, the image of $r \circ h$ 
acts by inner automorphisms on $G/N$.  Take $\ol{\alpha} \in \im h$.  It is easy to see that we can find a representative $\alpha \in \Aut(G)$ of $\ol{\alpha}$
that fixes $N$ pointwise and whose image in $\Aut(G/N)$ fixes $G/N$ pointwise.   If $g \in G$, then $\alpha(g) = gs$, for some $s \in N$.  
Since $\alpha$ fixes $N$, we see that $\alpha^i(g) = gs^i$.  Since $s \in N$, we know $s^{|N|}$ is trivial, so 
$\alpha^{|N|}$ is trivial.  Thus $\ol{\alpha}$ has prime-to-$p$ order, implying that $G_{L}/(\ker h)$ does as well.  
We conclude that the field of moduli $L'$ of $f$ relative to $L$ is a prime-to-$p$ extension of $L$.  
\end{proof}

For the next proposition, $K$ is a characteristic zero complete discrete valuation field with residue field $k$. 

\begin{prop}\label{Pstr2aux}
Let $f:Y \to X \cong \proj^1_K$ be a $G$-cover with bad reduction and stable model $f^{st}$ as in \S\ref{Sstable}.  
Suppose $G$ has a normal subgroup $N$ such that $p \nmid |N|$, and let 
$Z = Y/N$. Let $L/K$ be a finite extension such that both the stable model $\eta^{st}: Z^{st} \to X^{st}$ of 
$\eta: Z \to X$, and each of the 
branch points of the canonical map $q: Y \to Z$, can be defined over $L$.  Then the stable model $f^{st}$ of $f$ can be 
defined over a tame extension of $L$.
\end{prop}

\begin{proof}
By \cite[Remark 2.21]{Li:sr}, the minimal modification $(Z^{st})'$ of $Z^{st}$ that separates the specializations of 
the branch points of $q$ is defined over $L$.  
Note that $q$, being an $N$-cover, is tamely ramified.  We claim that $q^{st}: Y^{st} \to (Z^{st})'$ is defined over a tame extension of $L$
(along with the $N$-action).
  
The proof of the claim is almost completely contained in the proof of \cite[Th\'{e}or\`{e}me 3.7]{Sa:rm}, so we only give a sketch.
Break up the formal completion $\mc{Z}$ of $(Z^{st})'$ at its special fiber into three pieces: The piece $\mc{Z}_1$ is the disjoint union of the formal annuli 
corresponding to the completion of each double point; the piece $\mc{Z}_2$ is the disjoint union of the formal disks corresponding to the completion of the 
specialization of each branch point of $q$; 
and the piece $\mc{Z}_3$ is $\mc{Z} \backslash (\mc{Z}_1 \cup \mc{Z}_2)$.  Let $\hat{{Z}}_1$, $\hat{{Z}}_2$, and $\hat{{Z}}_3$ 
be the respective special fibers of $\mc{Z}_1$, $\mc{Z}_2$, and $\mc{Z}_3$.  The proof of
\cite[Th\'{e}or\`{e}me 3.7]{Sa:rm} shows how to lift the covers $q^{st}|_{\hat{{Z}}_1}$ and $q^{st}|_{\hat{{Z}}_3}$ to covers of $\mc{Z}_1$ and $\mc{Z}_3$, 
\'{e}tale on the generic fiber, after a possible tame extension of $L$.  Now, each connected component $\mc{C}_i$ of $\mc{Z}_2$ is isomorphic to  
$\text{Spf } S[[z_i]]$, where $S$ is the ring of integers of $L$.  The special fiber $\hat{{C}}_i$ of $\mc{C}_i$ is isomorphic to $\Spec k[[z_i]]$. 
The cover $q^{st}|_{\hat{{C}}_i}$ is given by a disjoint union of identical covers $\hat{{D}}_i \to \hat{{C}}_i$, each $\hat{{D}}_i$ being given by extracting a
$m_i$th root of $z_i$, where $m_i$ is the branching index of the branch point of $q$ specializing to $\hat{{C}}_i$.  
Since each branch point of $q$ is defined over $L$, there is a unique lift (over $L$) of $q^{st}|_{\hat{{C}}_i}$ to a cover of $\mc{C}_i$, \'{e}tale on the generic fiber
outside the 
appropriate point.  Using the arguments of \cite[Th\'{e}or\`{e}me 3.7]{Sa:rm}, the covers of $\mc{Z}_1$, $\mc{Z}_2$, and $\mc{Z}_3$ patch together uniquely to give a 
cover of $\mc{Z}$, defined over a tame extension of $L$.  By Grothendieck's existence theorem, this cover is algebraic, and it must be the base change of $q^{st}$.  
Thus $q^{st}$ is defined over a tame extension of $L$, and the claim is proved.

Let $M/L$ be a tame extension such that $q^{st}$ is defined over $M$.  By Lemma \ref{Lstr2aux2} 
applied to $q: Y \to Z$ and $\eta: Z \to X$, the field of moduli of $f$ is contained in some tame extension $M'$ of $M$.  
Since $M'$ has cohomological dimension 1, it follows (\cite[Proposition 2.5]{CH:hu}) that $f$ can be defined (as a 
$G$-cover) over $M'$.  Furthermore, $G_{M'} \leq G_M$ acts trivially on the special fiber $\ol{Y}$ of $Y^{st}$.  Thus $f^{st}$ is defined over $M'$.
\end{proof}

\begin{remark}\label{Rmodulitowers}
Suppose $f: Y \to X$, is a $G$-cover, $N \leq G$ is prime-to-$p$ and normal, and the field of \emph{moduli} of 
$f': Y' := Y/N \to X$ is $L$.  One can ask if this implies that the field of 
moduli of $f$ is a tame extension of $L$ (Proposition \ref{Pstr2aux} is the analogous statement for the minimal field of definition of the stable model).
If the answer to this question is yes, then some of the proofs in \S\ref{Smain} would be much easier.  Unfortunately, I believe the answer is no.
\end{remark}

\section{Proof of the main theorem}\label{Smain}
In this section, we will prove Theorem \ref{Tmain}.  Throughout \S\ref{Smain}, if $G \cong \ints/p^n \rtimes \ints/m$ and $p \nmid m$, then $Q_i$ $(0 \leq 
i \leq n)$ is the unique subgroup of order $p^i$.

Let $f: Y \to X = \proj^1$ be a three-point Galois cover defined over $\ol{\rats}$.  Our first step is to reduce to a local problem, which is the content of 
Proposition \ref {Plocal2global}. Let $\rats_p^{ur}$ be the \emph{completion} of the maximal unramified extension of $\rats$.
For an embedding $\iota: \ol{\rats} \hookrightarrow \ol{\rats_p^{ur}}$, let $f_{\iota}$ be the base change of $f$ to 
$\ol{\rats_p^{ur}}$ via $\iota$.
The following proposition shows that, for the purposes of Theorem \ref{Tmain}, we need only consider covers defined over 
$\ol{\rats_p^{ur}}$.

\begin{prop}\label{Plocal2global}
Let $K_{gl}$ be the field of moduli of $f$ (with respect to $\rats$) and let $K_{loc, \iota}$ 
be the field of moduli of
$f_{\iota}$ with respect to $\rats_p^{ur}$.  Fix $n \geq 0$, and suppose that for all embeddings $\iota$, the $n$th higher 
ramification groups of the Galois closure $L_{loc, \iota}$ of $K_{loc, \iota}/\rats_p^{ur}$ for the upper numbering vanish.  
Then all the $n$th higher ramification groups of the Galois closure $L_{gl}$ of $K_{gl}/\rats$ above $p$ for the upper numbering 
vanish.
\end{prop}

\begin{proof}
Pick a prime $q$ of $L_{gl}$ above $p$.  We will show that the $n$th higher ramification groups at $q$ vanish.  Choose a
place $r$ of $\ol{\rats}$ above $q$.  Then $r$ gives rise to an embedding 
$\iota_r: \ol{\rats} \hookrightarrow \ol{\rats_p^{ur}}$ preserving the higher ramification filtrations at $r$ for the
upper numbering (and the lower numbering).  Specifically, if $L/\rats_p^{ur}$ is a finite extension such that the $n$th
higher ramification group for the upper numbering vanishes, then the $n$th higher ramification group for the upper
numbering vanishes for $\iota_r^{-1}(L)/\rats$ at the unique prime of $\iota_r^{-1}(L)$ below $r$.  By assumption, the
$n$th higher ramification group for the upper numbering vanishes for $L_{loc, \iota_r}/\rats_p^{ur}$.
Also, the field $L' :=  \iota_r^{-1}(L_{loc, \iota_r})$ is Galois over $\rats$.
So if $K_{gl} \subseteq L'$, then $L_{gl} \subseteq L'$.  We know the $n$th higher ramification groups for $L'/\rats$
vanish.  We are thus reduced to showing that $K_{gl} \subseteq L'$.

Pick $\sigma \in G_{L'}$.  Then $\sigma$ extends by continuity to a unique automorphism $\tau$ in 
$G_{L_{loc, \iota_r}}$.  By the
definition of a field of moduli, $f_{\iota_r}^{\tau} \cong f_{\iota_r}$.  But then $f^{\sigma} \cong f$.  By the definition
of a field of moduli, $K_{gl} \subseteq L'$.
\end{proof}

So, in order to prove Theorem \ref{Tmain}, we can consider three-point covers defined over $\ol{\rats_p^{ur}}$.  In
fact, we generalize slightly, 
and consider three-point covers defined over algebraic closures of complete mixed characteristic
discrete valuation fields with algebraically closed residue fields.
In particular, throughout this section, $K_0$ is the fraction field of the ring $R_0$ of Witt vectors over $k$.  
On all extensions of $K_0$, we normalize the valuation $v$ so that $v(p) = 1$.
Also, write $K_n := K_0(\zeta_{p^n})$, with valuation ring $R_n$ (here $\zeta_{p^n}$ is a primitive $n$th root of unity).  
Let $G$ be a finite, $p$-solvable group with a cyclic $p$-Sylow subgroup $P$ of order $p^n$.
We assume $f:Y \to X = \proj^1$ is a three-point $G$-Galois cover of curves, branched at $0$, $1$, and $\infty$, \emph{a priori} defined
over some finite extension $K/K_0$.  Since $K$ has cohomological dimension 1, 
the field of moduli of $f$ relative to $K_0$ is the same as the minimal field of definition of $f$ that is an extension
of $K_0$ (\cite[Proposition 2.5]{CH:hu}).  
We will therefore go back and forth between fields of moduli and fields of definition without further notice.
Our default smooth model $X_R$ of $X$ is always the unique one such that the specializations of $0$, $1$, and $\infty$ do not collide
on the special fiber.  As in \S\ref{Sstable}, the stable model of $f$ is $f^{st}: Y^{st} \to X^{st}$ and the stable reduction is $\ol{f}: \ol{Y} \to \ol{X}$.
The original component of $\ol{X}$ will be denoted $\ol{X}_0$.

We will first prove Theorem \ref{Tmain} in the case that $G \cong \ints/p^n \rtimes \ints/m_G$.  The cases $m_G > 1$ and $m_G = 1$ have quite different flavors, and
we deal with them separately.  We in fact determine more than we need for Theorem \ref{Tmain}; namely, 
we determine bounds on the higher ramification filtrations of the extension $K^{st}/K_0$, where $K^{st}$ is the minimal field of definition of the stable model of $f$.  
In \S\ref{Spsolvable}, we will generalize to the $p$-solvable case.

\subsection{The case where $G \cong \ints/p^n \rtimes \ints/m_G$, $m_G > 1$}\label{Ssemidirect}

Let $G \cong \ints/p^n \rtimes \ints/m$ such that the conjugation action of $\ints/m$ is faithful (note that this implies $m = m_G$).  
We will show that the field of moduli with respect to $K_0$ of $f$ as a \emph{mere} cover is, in fact, $K_0$.
Then, we will show that its field of moduli with respect to $K_0$ as a $G$-cover is contained in $K_n$.  Lastly, we will
show that its stable model can be defined over a tame extension of $K_n$.

Let $\chi: \ints/m \to \FF_p^{\times}$ 
correspond to the conjugation action of $\ints/m$ on any order $p$ subquotient of $\ints/p^n$.
Now, there is an intermediate $\ints/m$-cover $\eta: Z \to X$ where $Z = Y/Q_n$. 
If $q: Y \to Z$ is the quotient map, then $f = \eta \circ q$.
Because it will be easier for our purposes here, let us assume that the three
branch points of $f$ are $x_1, x_2, x_3 \in R_0$, and that they have pairwise distinct
reduction to $k$ (in particular, none is $\infty$).
Since the $m$th roots of unity are contained in $K_0$, the cover $\eta$ can be
given birationally by the equation
$z^m = (x-x_1)^{a_1}(x-x_2)^{a_2}(x-x_3)^{a_3}$ with $0 \leq a_i < m$ for all $i
\in \{1, 2, 3\}$, where $a_1 + a_2 + a_3 \equiv 0 \pmod{m}$ and not all $a_i \equiv 0 \pmod{m}$.  
Since $\frac{g^*z}{z}$ is an $m$th root of unity, we can and do choose $z$ so that $g^*z = \chi(g)z$ for any $g \in \ints/m$.  
We know from Lemma \ref{Lmultdd} that the original component $\ol{X}_0$ is a $p^n$-component, and all
of the deformation data above $\ol{X}_0$ are multiplicative.

Consider the $\ints/p \rtimes \ints/m$-cover $f': Y' \to X$, where $Y' = Y/Q_{n-1}$.
The stable reduction $\ol{f}': \ol{Y}' \to \ol{X}'$ of this cover has a
multiplicative deformation datum $(\omega, \chi),$ over the original component $\ol{X}_0$.
For all $x \in \ol{X}_0$, recall that $(h_x, m_x)$ is the signature of the
deformation datum at $x$, and $\sigma_x =
h_x/m_x$ (see \S\ref{Sdefdata}).
Also, since there are no new tails (Proposition \ref{Pnoinsep}), 
\cite[Theorem 2, p.\ 992]{We:br} shows that the stable reduction $\ol{X}'$ consists
only of the original component $\ol{X}_0$ along with a primitive \'{e}tale tail
$\ol{X}_i$ for each branch point $x_i$ of $f$ (or $f'$) with prime-to-$p$ ramification index.  
The tail $\ol{X}_i$ intersects $\ol{X}_0$ at the specialization of $x_i$ to $\ol{X}_0$.

\begin{prop}\label{Pdiffdiv}
For $i = 1, 2, 3$, let $\ol{x}_i$ be the specialization of $x_i$ to $\ol{X}_0$.
For short, write $h_i$, $m_i$, and $\sigma_i$ for $h_{\ol{x}_i}$, $m_{\ol{x}_i}$,
and $\sigma_{\ol{x}_i}$.
\begin{enumerate}[(i)]
\item For $i = 1, 2, 3$, $h_i \equiv a_i/\gcd(m, a_i) \pmod{m_i}$.
\item In fact, the $h_i$ depend only on the $\ints/m$-cover $\eta:Z \to X$.
\end{enumerate}
\end{prop}

\begin{proof}
\emph{To (i):} (cf.\ \cite[Proposition 2.5]{We:mc}) 
Let $\ol{Z}_0$ be the unique irreducible component lying above $\ol{X}_0$, and
suppose that $\ol{z}_i \in \ol{Z}_0$ lies above $\ol{x}_i$.  
Let $t_i$ be the formal parameter at $\ol{z}_i$ given by $z^{\alpha}(x-x_i)^{\beta}$, where $\alpha a_i + \beta m = 
\gcd(m, a_i)$.  Then 
$$\omega = (c_0 t_i^{h_i-1} + \sum_{j=1}^{\infty}c_j t_i^{h_i - 1 + j}) dt_i$$
in a formal neighborhood of $\ol{z}_i$.
Recall that, for $g \in \ints/m$, $g^*z = \chi(g)z$ and $g^*\omega = \chi(g)\omega$.
Then $$\chi(g) = \frac{g^*\omega}{\omega} = \left(\frac{g^*t_i}{t_i}\right)^{h_i} = \left(\frac{g^*z}{z}\right)^{\alpha h_i} = \chi(g^{\alpha h_i}).$$
So $\alpha h_i \equiv 1 \pmod{m}$.  It follows that $h_i \gcd(m, a_i) \equiv h_i (\alpha a_i + \beta m) \equiv a_i \pmod{m}$. 
It is clear that the ramification index $m_i$ at $\ol{x}_i$ is $m/\gcd(m, a_i)$.
Dividing $$h_i \gcd(m, a_i) \equiv a_i \pmod{m}$$ by $\gcd(m, a_i)$ yields (i).
\\
\\
\emph{To (ii):} Since we know the congruence class of $h_i$ modulo $m_i$, it
follows that the fractional part $\langle
\sigma_i \rangle$ of $\sigma_i$ is determined by $\eta:Z \to X$.  
But if $x_i$ corresponds to a primitive tail, the vanishing cycles formula (\ref{Evancycles}) shows that $0 < \sigma_i < 1$. 
If $x_i$ corresponds to a wild branch point, then $\sigma_i = 0$.
Thus $\sigma_i$ is determined by $\langle \sigma_i \rangle$, so it is determined
by $\eta:Z \to X$.
Since $h_i = \sigma_im_i$, we are done.
\end{proof}

\begin{corollary}\label{Cdiffdiv}
The differential form $\omega$ corresponding to the cover $f': Y' \to X$ is
determined (up to multiplication by an element of $\FF_p^{\times}$) by $\eta:Z \to X$.
\end{corollary}

\begin{proof}
Proposition \ref{Pdiffdiv} determines the divisor corresponding to $\omega$ from
$\eta:Z \to X$.  Two meromorphic differential forms on a complete curve can have the same divisor
only if they differ by a scalar multiple.  Also, if $\omega$ is logarithmic and $c \in k$, then $c\omega$ is logarithmic
iff $c \in \FF_p$, by basic properties of the Cartier operator (see, for instance, \cite[p.\ 136]{We:mc}).
\end{proof}

We will now show that $\eta:Z \to X$ determines not only the differential form
$\omega$, but also the entire cover $f:Y \to X$ as a mere cover.  This is the key lemma of this section.
We will prove it in several stages, using an induction.

\begin{lemma}\label{Ldiffcov}
Assume $m > 1$. 
\begin{enumerate}[(i)]
\item
If $f:Y \to X$ is a three-point $\ints/p^n \rtimes \ints/m$-cover (with faithful
conjugation action of $\ints/m$ on
$\ints/p^n$) defined over some finite extension $K/K_0$, 
then it is determined as a mere cover by the map $\eta: Z = Y/Q_n \to X$.
\item
If $f:Y \to X$ is a three-point $\ints/p^n \rtimes \ints/m$-cover (with faithful
conjugation action of $\ints/m$ on
$\ints/p^n$) defined over some finite extension $K/K_0$, its field
of moduli (as a mere cover) with respect to $K_0$ is $K_0$, and $f$ can be defined over $K_0$ (as a mere
cover).
\item
In the situation of part (ii), the field of moduli of $f$ (as a $\ints/p^n
\rtimes \ints/m$-cover) with respect to $K_0$ is contained in $K_n = K_0(\zeta_{p^n})$.  
Thus $f$ can be defined over $K_n$ (as a $\ints/p^n \rtimes \ints/m$-cover).
\end{enumerate}
\end{lemma}

\begin{proof}
\emph{To (i):} 
We first assume $n=1$, so $G \cong \ints/p \rtimes \ints/m$.
Write $Z^{st}$ for $Y^{st}/Q_1$ and $\ol{Z}$ for the special fiber of $Z^{st}$.
We know from Corollary \ref{Cdiffdiv} that $\eta$ determines (up to a scalar
multiple in $\FF_p^{\times}$) the logarithmic differential form 
$\omega$ that is part of the deformation datum $(\ol{Z}_0, \omega)$ on the irreducible
component $\ol{Z}_0$ above $\ol{X}_0$. 
Let $\xi$ be the generic point of $\ol{Z}_0$.
Then $\omega$ is of the form $d\ol{u}/\ol{u}$, where $\ol{u} \in k(\ol{Z}_0)$ is
the reduction of some function $u \in
\hat{\mc{O}}_{(Z')^{st}, \xi}$.
Moreover, by \cite[Ch.\ 5, D\'{e}finition 1.9]{He:ht}, we can choose $u$ such that the cover $Y \to Z$ is given
birationally by extracting a $p$th root of $u$ 
(viewing $u \in K(Z) \cap \hat{\mc{O}}_{(Z')^{st}, \xi}$).
That is, $K(Y) = K(Z)[t]/(t^p - u)$.
We wish to show that knowledge of $d\ol{u}/\ol{u}$ up to a scalar multiple $c
\in \FF_p^{\times}$ determines $u$ up to raising to 
the $c$th power, and then possibly multiplication by a $p$th power in $K(Z)$ (as
this shows $Y' \to X$ is uniquely
determined as a mere cover). 
This is equivalent to showing that knowledge of $d\ol{u}/\ol{u}$
determines $u$ up to a $p$th power (i.e., that if $d\ol{u}/\ol{u} = d\ol{v}/\ol{v}$, then $u/v$ is a $p$th power in $K(Z)$).

Suppose that there exist $u, v \in K(Z) \cap \mc{O}_{(Z')^{st}, \xi}$ such
that $d\ol{u}/\ol{u} = d\ol{v}/\ol{v}$.
Then $\ol{u} = \ol{\kappa} \, \ol{v}$, with $\ol{\kappa} \in k(\ol{Z}_0)^p$.  Since $\ol{\kappa}$ is a $p$th power, 
it lifts to some $p$th power $\kappa$ in $K$.  Multiplying $v$ by
$\kappa$, we can assume that $\ol{u} = \ol{v}$.
Consider the cover $Y' \to Z$ given birationally by the field extension $K(Y')
= K(Z)[t]/(t^p - u/v)$.  
Since $\ol{u} = \ol{v}$, we have that $u/v$ is congruent to $1$ in the residue field of $\mc{O}_{(Z')^{st}, \xi}$.
This means that the cover $Y' \to Z$ cannot have multiplicative reduction (see
\cite[Ch.\ 5, Proposition 1.6]{He:ht}).  
But the cover $Y' \to Z \to X$ is a $\ints/p \rtimes \ints/m$-cover, branched
at three points, so it must have
multiplicative reduction if the $\ints/p$ part is nontrivial (Lemma \ref{Lmultdd}).
Thus it is trivial, which means that $u/v$ is a $p$th power in $K(Z)$, i.e., $u
= \phi^pv$ for some $\phi \in K(Z)$.  This proves the case $n=1$.

For $n > 1$, we proceed by induction.  
We assume that (i) is known for $\ints/p^{n-1} \rtimes
\ints/m$-covers.  Given $\eta: Z \to X$, we wish to determine $u \in
K(Z)^{\times}/(K(Z)^{\times})^{p^n}$ such that 
$K(Y)$ is given by $K(Z)[t]/(t^{p^n} - u)$.  
By the induction hypothesis, we know that $u$ is well-determined up to
multiplication by a $p^{n-1}$st power.
Suppose that extracting $p^n$th roots of $u$ and $v$ both give $\ints/p^n
\rtimes \ints/m$-covers branched at $0, 1,$ and $\infty$.  
Consider the cover $Y' \to Z \to X$ of smooth curves given birationally by $K(Y') =
K(Z)[t]/(t^{p^n} - u/v)$.
Since $u/v$ is a $p^{n-1}$st power in $K(Z)$, this cover splits into a disjoint
union of $p^{n-1}$ different
$\ints/p \rtimes \ints/m$-covers.
By our previous argument, each of these covers can be given by extracting a $p$th root of some
power of $u$ itself!   So $\sqrt[p^{n-1}]{u/v} = u^cw^p$, where $w \in K(Z)$ and $c \in \ints$.  
Thus $v = u^{1 - p^{n-1}c}w^{-p^n}$, which means that extracting $p^n$th roots of
either $u$ or $v$ gives the same mere cover.
\\
\\
\emph{To (ii):}
We know that the cyclic cover $\eta$ of part (i) is defined over $K_0$, because we
have written it down explicitly.
Now, for $\sigma \in G_{K_0}$, $f^{\sigma}$ is a $\ints/p^n \rtimes
\ints/m$-cover with quotient cover $\eta$, branched at
0, 1, and $\infty$.
By part (i), there is only one such (mere) cover, so $f^{\sigma} \cong f$ as mere
covers.
So the field of moduli of $f$ as a mere cover with respect to $K_0$ is $K_0$.
It is also a field of definition, by \cite[Proposition 2.5]{CH:hu}.
\\
\\
\emph{To (iii):}
Since $f$ is defined over $K_0$ as a mere cover, it is certainly defined over
$K_n$ as a mere cover.  We thus obtain a homomorphism $h: G_{K_n} \to \Out(G)$, as
in \S\ref{Smere}.
By Kummer theory, we can write $K_n(Z) \hookrightarrow K_n(Y)$ as a Kummer
extension, with Galois action defined over
$K_n$.  The means that the image of $h$ acts trivially on $\ints/p^n$.  
Furthermore, $\eta: Z \to X$ is defined over $K_0$ as a $\ints/m$-cover.  Thus, if $r: \Out(G) \to \Out(\ints/m)$ is the
natural map, the image of $r \circ h$ acts trivially on $\ints/m$.  
But the only automorphisms of $G$
satisfying both of these properties are inner, so $h$ is trivial.  This shows that the field of moduli of $f$ with respect to $K_0$ is $K_n$.
Since $K_0$ has cohomological dimension 1, we see that $f:Y \to X$ is defined
over $K_n$ as a $\ints/p^n \rtimes \ints/m$-cover.
\end{proof}

We know from Lemma \ref{Ldiffcov} that $f$ is defined over $K_0$ as a mere
cover and over $K_n$ as a $G$-cover.
Recall from \S\ref{Sstable} that the minimal field of
definition of the stable model $K^{st}$ is the fixed field
of the subgroup $\Gamma^{st} \leq G_{K_0}$ that acts trivially on the
stable reduction $\ol{f}: \ol{Y} \to \ol{X}$.
Recall also that the action of $G_{K_n}$ centralizes the action of $G$.

\begin{lemma}\label{Lfixallcomponents}
If $g \in G_{K_n}$ acts on $\ol{Y}$ with order $p$, then $g$ acts trivially on
$\ol{Y}$.
\end{lemma}

\begin{proof}
First, note that since each tail $\ol{X}_b$ of $\ol{X}$ is primitive (Proposition \ref{Pnoinsep}), each
contains the specialization of a $K_0$-rational point (which must be fixed by $g$).  
As in the proof of Proposition \ref{Pstabletail}, $g$ fixes all of $\ol{X}$ pointwise.  

There are at least two primitive tails, because, for $G \cong \ints/p^n \rtimes \ints/m$ with $m > 1$ and faithful conjugation action, a three-point
$G$-cover must have at least two branch points with prime-to-$p$ branching index.  Since $G$ has trivial center, then 
\cite[Lemmas 5.4 and 5.8]{Ob:vc} shows that $g$ acts trivially on $\ol{X}$. 
\end{proof}

\begin{prop}\label{Psolvstabtame}
Assume $m > 1$.
Let $f:Y \to X$ be a three-point $G$-cover, where $G \cong \ints/p^n \rtimes \ints/m$ 
(with faithful conjugation action of $\ints/m$ on $\ints/p^n$).  Choose a model for $f$ over $K_n$ (Lemma
\ref{Ldiffcov} (iii)).  Then there is a tame extension $K^{stab}/K_n$ such that
the stable model $f^{st}: Y^{st} \to X^{st}$ is defined over $K^{stab}$.  In particular, the $n$th higher ramification groups for the upper
numbering of $K^{stab}/K_0$ vanish.
\end{prop}

\begin{proof}
By Lemma \ref{Lfixallcomponents}, no element of $G_{K_n}$ acts with order $p$ on $\ol{Y}$.  So the subgroup of $G_{K_n}$ that acts trivially on $\ol{Y}$ has
prime-to-$p$ index, and its fixed field $K^{stab}$ is a tame extension of $K_n$.  
By \cite[Corollary to IV, Proposition 18]{Se:lf}, the $n$th higher ramification groups for the upper numbering of the
extension $K_n/K_0$ vanish.  By Lemma \ref{Ltamenochange}, the $n$th higher ramification groups for the upper numbering
of $K^{stab}/K_0$ vanish.
\end{proof}

\subsection{The case where $G \cong \ints/p^n$}\label{Smequals1}

Maintaining the notation of \S\ref{Smain}, we now set $G \cong \ints/p^n$.  
Finding the field of moduli is easy in this case, but understanding the stable model (which is
needed to apply Proposition \ref{Pstr2aux}) is more difficult.

\begin{prop}\label{Pcyclicmoduli}
The field of moduli of $f: Y \to X$ relative to $K_0$ is $K_n = K_0(\zeta_{p^n})$.
\end{prop}
\begin{proof} Since the field of moduli of $f$ relative to $K_0$ is the intersection of all extensions of $K_0$ which are
fields of definition of $f$, it suffices to show that $K_n$ is the minimal such extension.
By Kummer theory, $f$ can be defined over $\ol{K_0}$ birationally by the
equation $y^{p^n} = x^a(x-1)^b$, for some integral $a$ and $b$.  The Galois action is generated by $y \mapsto
\zeta_{p^n} y$.  This cover is clearly defined over $K_n$ as a $G$-cover.
 
Since $Y$ is connected, $f$ is totally ramified above at least one of the branch points $x_0$ (i.e., with index $p^n$).  
Let $y_0 \in Y$ be the
unique point above $x_0$.  Assume $f$ is defined over some finite extension $K/K_0$ as a
$G$-cover, where $Y$ and $X$ are considered
as $K$-varieties. 
Then, by \cite[Proposition 4.2.11]{Ra:sp}, the residue field $K(y_0)$ of $y_0$
contains the $p^n$th roots of unity.  Since
$y_0$ is totally ramified, $K(y_0) = K(x_0) = K$, and thus $K \supseteq K_n$.  So $K_n$ is
the minimal extension of $K_0$ which is a field of definition of $f$.  Thus $K_n$ is the field of moduli of $f$ with
respect to $K_0$.
\end{proof}

In the rest of this section, we analyze the stable model of three-point $G$-covers $f: Y \to X$ (a complete description, at least 
in the case $p > 3$, is given in Appendix \ref{Aexplicit}).
By Kummer theory, $f$ can be given (over $\ol{K_0}$) by an equation of the form $y^{p^n} = cx^a(x-1)^b$, 
for any $c \in \ol{K_0}$ (note that different values of $c$ might give different models over subfields of $\ol{K_0}$).  
The ramification indices above $0$, $1$, and $\infty$ are $p^{n-v(a)}$, $p^{n-v(b)}$, and $p^{n- v(a+b)}$, respectively.  
Since $Y$ is connected, we must have that at least two of $a, b$, and $a+b$ are prime to $p$.  Note that if, $p=2$, then exactly 
two of $a, b$, and $a+b$ are prime to $p$.
In all cases, we assume without loss of generality that $f$ is totally ramified above $0$ and $\infty$, and we set $s$ to be such 
that $p^s$ is the ramification index above $1$.  Then $v(b) = n - s$.  

As in \S\ref{Sstable}, write $f^{st}: Y^{st} \to X^{st}$ for the stable model of $f$, and
$\ol{f}:\ol{Y} \to \ol{X}$ for the stable reduction.  

\begin{lemma}[cf.\ \cite{CM:sr}, \S3]\label{Lonenewtail}
The stable reduction $\ol{X}$ (over $\ol{K_0}$)
contains exactly one \'{e}tale tail $\ol{X}_b$, which is a new tail with effective invariant 
$\sigma_b = 2$.  

If $p>3$, or $p = 3$ and either $s > 1$ or $s = n = 1$, set $d = \frac{a}{a+b}$.  If $p=3$ and $n > s =1$, set 
$$d = \frac{a}{a+b} + \frac{\sqrt[3]{3^{2n+1}\binom{b}{3}}}{a+b},$$ where we choose any cube root.
If $p = 2$, set $$d = \frac{a}{a+b} + \frac{\sqrt{2^nbi}}{(a+b)^2},$$ where 
$i^2 = -1$ and the square root sign represents either square root.

Then $\ol{X}_b$ corresponds to the disk of radius $|e|$ centered at $d$, where $e \in \ol{K_0}$ has $v(e) = \frac{1}{2}(2n - s + \frac{1}{p-1})$.
\end{lemma}

\begin{proof} 
By the Hasse-Arf theorem, the effective ramification invariant $\sigma$ of any \'{e}tale tail is an integer.  
Clearly there are no primitive tails, as there are no branch points with prime-to-$p$ branching index.  By Lemma 
\ref{Ltailbounds}, any new tail has $\sigma \geq 2$.  By the vanishing cycles formula (\ref{Evancycles}), there is exactly 
one new tail $\ol{X}_b$ and its invariant $\sigma_b$ is equal to $2$. 

We know that $f$ is given by an equation of the form $y^{p^n} = g(x) := cx^a(x-1)^b$, and that any value of $c$ yields $f$ 
over $\ol{K_0}$.  Taking $K$ sufficiently large, we may (and do) assume that $c = d^{-a}(d-1)^{-b}$.  Note that, in all cases, $g(d) = 1$, 
$v(d) = v(a) = 0$, and $v(d-1) = v(b) = n - s$.

Let $K$ be a subfield of $\ol{K_0}$
containing $K_0(\zeta_{p^n}, e, d)$.  Let $R$ be the valuation ring of $K$. 
Consider the smooth model $X'_R$ of $\proj^1_K$ corresponding to the coordinate $t$, where $x = d + et$. 
The formal disk $D$ corresponding to the
completion of $D_k := X'_k \backslash \{t = \infty\}$ in $X'_R$ is the
closed disk of radius 1 centered at $t = 0$, or equivalently, the disk of radius $|e|$ centered at $x = d$ (\S\ref{Sdisks}). 
Its ring of functions is $R\{t\}$.  

In order to calculate the normalization of $X'_R$ in $K(Y)$, we calculate the normalization $E$ of $D$ in 
the fraction field of $R\{t\}[y]/(y^{p^n} - g(x)) = R\{t\}[y]/(y^{p^n} -
g(d+et))$.  Now, $g(d+et) = \sum_{\ell=0}^{a+b} c_{\ell}t^{\ell},$ where 
\begin{equation}\label{Ecell}
c_{\ell} = e^{\ell} \sum_{j=0}^{\ell} \binom{a}{\ell-j}\binom{b}{j}d^{j-\ell}(d-1)^{-j}.
\end{equation}  

If $s = n$ and $\ell \geq 3$, then clearly $v(c_{\ell}) \geq v(e^{\ell}) = \frac{\ell}{2}(n + \frac{1}{p-1}) > n + \frac{1}{p-1}$.
If $s < n$ and $\ell \geq 3$, then the $j = \ell$ term is the term of least valuation in (\ref{Ecell}), and thus it has the same valuation as $c_{\ell}$.  
We obtain
\begin{equation}\label{Eclval}
v(c_{\ell}) = \ell v(e) + v(b) - v(\ell) - \ell(n-s) = n + \frac{1}{p-1} + \frac{\ell-2}{2}(s + \frac{1}{p-1}) - v(\ell)
\end{equation}
(unless, of course, $c_{\ell} = 0$).

Now, assume either that $p > 3$, or if $p=3$, then $s > 1$ or $s = n = 1$.  Then $d = \frac{a}{a+b}$.  
It is easy to check, using (\ref{Eclval}), that $v(c_{\ell}) > n + \frac{1}{p-1}$ for $\ell \geq 3$.
Equation (\ref{Ecell}) shows that $c_0 = 1$, $c_1 = 0$, and $c_2 = \frac{(a+b)^3}{ab}e^2$, which has valuation $n+\frac{1}{p-1}$.
Thus we are in the situation (i) of Lemma \ref{Lartinschreier} (with $h=2$), and
the special fiber $E_k$ of $E$ is a disjoint union of $p^{n-1}$ \'{e}tale covers of $D_k \cong
\aff^1_k$.  Each of these extends to an Artin-Schreier cover of conductor 2 over $\proj^1_k$.  
By Lemma \ref{Lartinschreiergenus}, these have genus $\frac{p-1}{2} > 0$.  Therefore, by Lemma \ref{Lgenusincluded}, 
the component $\ol{X}_b$ corresponding to $D$ is included in the stable model.  By Lemma \ref{Letaletail}, it is a tail.
Since there is only one tail of $\ol{X}$, and it has effective ramification invariant $2$, it must correspond to $\ol{X}_b$.

For the cases where either $p=2$, or $p = 3$ and $n > s = 1$, see Lemma \ref{Lsmallprimetails}.
\end{proof}

\begin{remark}\label{Rcolemanmccallum}
The computation of Lemma \ref{Lonenewtail}
is similar to the relevant parts of \cite[\S3]{CM:sr} (in particular, Lemma 3.6 and Case 5 of Theorem 3.18 there).  Our task is simplified because
we know from the outset what we are looking for, that is, a new tail with $\sigma_b = 2$.
\end{remark}

\begin{corollary}\label{Cnoinseptails}
\begin{enumerate}[(i)]
\item If $f$ is totally ramified above $\{0, 1, \infty\}$, then $\ol{X}$ has no inseparable tails.  
\item If $p > 3$, and $f$ is totally ramified above only $\{0, \infty\}$, then if $\ol{X}$ has an inseparable tail, the tail 
contains the specialization of $x=1$.
\item If $p = 3$, suppose $f$ is totally ramified above only $\{0, \infty\}$ and ramified of index $3^s$ above $1$.  Then any inseparable tail 
of $\ol{X}$ not containing the specialization of $x=1$ is a $p^{s-1}$-tail (in particular, $s \geq 2$).  Furthermore, such a tail corresponds to the
disk of radius $|e'|$ centered at $d'$, where $v(e') = n - s + \frac{2}{3}$ and $d' = \frac{a}{a+b} + \frac{\sqrt[3]{3^{2(n-s+1) + 1}\binom{b}{3}}}{a+b}$. 
\item If $p = 2$, suppose $f$ is totally ramified above only $\{0, \infty\}$ and ramified of index $2^s$ above $1$.  Then any 
inseparable tail of $\ol{X}$ not containing the specialization of $x=1$ is a $p^j$-tail, for some $j<s$.  Furthermore, such a 
tail corresponds to the disk of radius $|e_j|$ centered at $d_j$, where $v(e_j) = \frac{1}{2}(2n-s-j+1)$, $$d_j = \frac{a}{a+b} + 
\frac{\sqrt{2^{n-j}bi}}{(a+b)^2},$$ and $i^2 = -1$.
\end{enumerate}
\end{corollary}

\begin{proof} 
\emph{To (i):}  Let $d = \frac{a}{a+b}$ as in Lemma \ref{Lonenewtail}.
Suppose there is an inseparable $p^j$-tail $\ol{X}_c \subset \ol{X}$ (we know $j < n$, by Lemma \ref{Ltailetale}).
By Proposition \ref{Pcorrectspec}, $\ol{X}_c$ is a new inseparable tail.  By Lemma \ref{Linsep2etale},
$\ol{X}_c$ is a new (\'{e}tale) tail of the stable reduction of $Y/Q_j \to X$. 
Its corresponding disk must contain $d$, by Lemma \ref{Lonenewtail} (substituting $n-j$ for $n$ in the statement).  
But this is absurd, because the disks corresponding to $\ol{X}_c$ and the \'{e}tale tail $\ol{X}_b$ are disjoint.
\\
\\
\noindent\emph{To (ii):}
Assume $f$ is ramified above $x=1$ of index $p^s$, $s < n$.  
Let $\ol{X}_c$ be a new inseparable $p^j$-tail of $\ol{X}$, and $\sigma_c$ its ramification 
invariant.  By Lemma \ref{Ltailbounds}, $\sigma_c > 1$.  Let $\ol{Y}_c$ be a component of $\ol{Y}$ lying above $\ol{X}_c$.
If $j \geq s$, we see that $Y/Q_j \to X$ is branched at two points,
and thus has genus zero.  Since $Q_j \leq I_{\ol{Y}_c}$, the constancy of arithmetic genus in flat families shows that $\ol{Y}_c$ 
has genus zero.  But any component $\ol{Y}_c$ above $\ol{X}_c$ must have genus greater than 1 (see \cite[Lemme
1.1.6]{Ra:sp}).  This is a contradiction.

Now suppose $j < s$.  Then $Y/Q_j \to X$ is a three-point cover.  So we obtain the same contradiction as in (i).
\\
\\
\emph{To (iii):}
As in (ii), we see that any new inseparable $p^j$-tail $\ol{X}_c$ of $\ol{X}$ must satisfy $j < s$.  In particular, $s \geq 2$.  
As in (i), $\ol{X}_c$ must correspond to the same disk as the new \'{e}tale tail of the stable reduction of $f': Y/Q_j \to X$,
but the disk must not contain the specialization of $d=\frac{a}{a+b}$.   By Lemma \ref{Lonenewtail}, this can only happen if $f'$ has  
degree greater than $3$, but is branched of index $3$ above $1$.  Thus $j = s-1$.  
Thus $f'$ is a $\ints/p^{n-s+1}$-cover.  We conclude using Lemma \ref{Lonenewtail}, replacing $n$ by $n-s+1$ and $s$ by $1$.
\\
\\
\emph{To (iv):}
Let $j$ and $\ol{X}_c$ be as in part (ii).
As in (ii), we may assume $j < s$.  
As in (i), $\ol{X}_c$ is the new \'{e}tale tail of the stable reduction of $f': Y/Q_j \to X$.  
The cover $f'$ is a $\ints/p^{n-j}$-cover, totally ramified above $0$ and $\infty$, and ramified of index $2^{s-j}$ above $1$.
We conclude using Lemma \ref{Lonenewtail}, replacing $n$ by $n-j$ and $s$ by $s-j$.
\end{proof}

\begin{corollary}\label{Cinsepprimitive}
In cases (ii), (iii), and (iv) above, $x=1$ in fact specializes to an inseparable tail.
\end{corollary}

\begin{proof}
If $x=1$ specializes to a component $\ol{W}$ that is not a tail, then there exists a tail $\ol{X}_c$ lying outward from $\ol{W}$.  
If $\ol{X}_c$ is a $p^i$-tail, then Lemma \ref{Ltailetale} and Proposition \ref{Pcorrectspec} show that $i < s$.    
By Lemma \ref{Linsep2etale}, $\ol{X}_c$ is an \'{e}tale tail of the stable model of $Y/Q_i \to X$.  As $i < s$, this is still a three-point cover.  
So we may assume (still, for the sake of contradiction) that there is an \'{e}tale tail lying outwards from the specialization of $x=1$.  By Lemma \ref{Lonenewtail}, we have 
$\sigma_c = 2$, and $\ol{X}_c$ is the only \'{e}tale tail of $f$.

Let $e_0$ (resp.\ $e_1$) be the edge of $\mc{G}'$ with 
source corresponding to $\ol{W}$ and target corresponding to the branch point $x=1$ (resp.\ the immediately
following component of $\ol{X}$ in the direction of $\ol{X}_c$).  Then $\sigma^{\eff}_{e_1} = 2$ by Lemma \ref{Lsigmaeff}, and $\sigma^{\eff}_{e_0} = 0$.
The deformation data above $\ol{W}$ are multiplicative and identical, and $\sigma^{\eff}$ is given by a weighted 
average of invariants.  So for any deformation datum $\omega$ above $\ol{W}$, we have $\sigma_{x_0} = 0$ and $\sigma_{x_1} = 2$, 
where the points $x_0$ and $x_1$ correspond to $e_0$ and $e_1$, respectively.  Furthermore, $\sigma_x = 1$ for all $x$ other than $x_0$, $x_1$, and the intersection point
$x_2$ of $\ol{W}$ and the next most inward component.

Now, by a similar argument as in the first part of the proof of Corollary \ref{Cnoinseptails} (ii), any component of $\ol{Y}$ above $\ol{W}$ must 
have genus zero.  Thus $\omega$ has degree $-2$.  Since $\omega$ has simple poles above $x_0$ and simple zeroes above $x_1$, it 
must have a double pole above $x_2$. But a logarithmic 
differential form cannot have a double pole.  This is a contradiction.
\end{proof}

\begin{remark}\label{Rallinsep}
In Corollary \ref{Cnoinseptails} (iv), there in fact does exist an inseparable $p^j$-tail for each $1 \leq j < s$.
Each of these is the same as the unique new tail of the cover $Y/Q_j \to X$.  We omit the details.
\end{remark}

We give the major result of this section:
\begin{prop}\label{Pm1stable}
Assume $G = \ints/p^n$, $n \geq 1$, and $f:Y \to X$ is a three-point $G$-cover defined over $\ol{K_0}$, totally ramified above $\{0, \infty\}$
and ramified of index $p^s$ above $1$.  Suppose $f$ is given over $\ol{K_0}$ by $y^{p^n} = x^a(x-1)^b$.
\begin{enumerate}[(i)]
\item If $s=n$ (i.e., $f$ is totally ramified above $1$), then there is a model for $f$ defined over $K_n =
K_0(\zeta_{p^n})$ whose stable model can be defined over a tame extension $K^{stab}/K_n$.
\item If $p > 3$ and $s < n$, then there is a model for $f$ over $K_n$ whose stable model can be defined over a tame extension 
$K^{stab}/K_n\left(\sqrt[p^{n-s}]{\frac{a}{a+b}}\right)$.  
\item If $p = 3$ and $1 =s < n$, then there is a model for $f$ over 
$K_n\left(\sqrt[3]{3^{2n+1}\binom{b}{3}}\right)$ whose stable model can be defined over a tame extension $K^{stab}$ of 
$$K_n\left(\sqrt[3]{3^{2n+1}\binom{b}{3}}, \ \ \sqrt[3^{n-1}]{\frac{a}{a+b}}\right).$$
\item Assume $p = 3$ and $1 < s < n$.  Let $$d' = \frac{a}{a+b} + \frac{\sqrt[3]{3^{2(n-s+1) + 1}\binom{b}{3}}}{a+b}.$$ 
Then there is a model for $f$ over $K_n$ whose stable model can be defined over a tame extension $K^{stab}$ of
$$K_n\left(d', \sqrt[3^{n-s}]{\frac{a}{a+b}}, \ \ \sqrt[3^{n-s+1}]{\frac{(d')^a(d'-1)^b}{a^ab^b(a+b)^{-(a+b)}}}\right).$$  
\item Assume $p = 2$. 
For $0 \leq j < s$, let $$d_j = \frac{a}{a+b} + \frac{\sqrt{2^{n-j}bi}}{(a+b)^2},$$ where $i^2 =-1$, and the square root sign represents either square root.
Then there is a model for $f$ over $K_n$ whose stable model can defined over a tame extension $K^{stab}$ of 
$$K := K_n\left(\sqrt[2^{n-1}]{d_0}, \sqrt[2^{s-1}]{d_0 -1}, \sqrt[2^{n-j}]{d_j}, \sqrt[2^{s-j}]{d_j-1}\right)_{1 \leq j < s}.$$
\end{enumerate}
\end{prop}

\begin{proof}  In each case of the proposition, let $d$ be as in Lemma \ref{Lonenewtail}.  Set $c = d^{-a}(d-1)^{-b}$.  The model of $f$ 
we will use will always be the one given by the equation $y^{p^n} = cx^a(x-1)^b$.  In all cases, there is a unique \'{e}tale tail
$\ol{W}$ of $\ol{X}$ containing the specialization of $x=d$, which is a smooth point of $\ol{X}$.
Furthermore, the points in the fiber of $f$ above $x=d$ are all $K_n$-rational.
\\\\
\emph{To (i):}
Since $s = n$, we have $d = \frac{a}{a+b}$ and $a$, $b$, $a+b$ are prime to $p$.  Our model for $f$ is defined over $K_n$.
By Corollary \ref{Cnoinseptails}, the tail $\ol{W}$ is the unique tail of $\ol{X}$.  Since the point $x=d$ and all points in the fiber
of $f$ above $x=d$ are $K_n$-rational, their specializations are fixed by $G_{K_n}$.  
By Proposition \ref{Pstabletail}, the stable model of $f$ is defined over a tame extension of $K_n$.
\\
\\
\emph{To (ii) and (iii):}
By Corollary \ref{Cnoinseptails} (ii) and (iii), there is a unique inseparable tail $\ol{W}'$ 
containing the specialization of $x=1$ (to a smooth point of $\ol{X}$).  
Now, consider $Y/Q_s$. (note that $Q_s$ is the inertia group above $x=1$).  
This is a cover of $X$ given birationally by the equation $y^{p^{n-s}} = cx^a(x-1)^b$.  Since
$p^{n-s}$ exactly divides $b$, we set $y' = y/(x-1)^{b/p^{(n-s)}}$.  The new
equation $(y')^{p^{n-s}} = cx^a$ shows that the points above $x=1$ in $Y/Q_s$ are defined over the field $K_{n-s}(c, \sqrt[p^{n-s}]{c}) = K_{n-s}(d, \sqrt[p^{n-s}]{d})$, and their specializations are thus fixed by its absolute Galois group. 
Since the map $Y^{st} \to Y^{st}/Q_s$ is radicial above $\ol{W}'$, all points of $\ol{Y}$ above the specialization of $x=1$
are fixed by $G_{K_{n-s}(d, \sqrt[p^{n-s}]{d})}$.
By Proposition \ref{Pstabletail}, the stable model of $f$ is defined over a tame extension of $K_n(d, \sqrt[p^{n-s}]{d})$.  

If $p > 3$ and $s < n$, then $K_n(d, \sqrt[p^{n-s}]{d}) = K_n(\sqrt[p^{n-s}]{\frac{a}{a+b}})$, finishing the proof of (ii).  If $p = 3$ and $s = 1$, then 
$d = \frac{a}{a+b}\left(1 + \frac{\sqrt[3]{3^{2n+1}\binom{b}{3}}}{a}\right)$.  Since $v(\sqrt[3]{3^{2n+1}\binom{b}{3}}) =  n - \frac{1}{3}$, 
the binomial theorem shows that $1 + \frac{\sqrt[3]{3^{2n+1}\binom{b}{3}}}{a}$ is a $3^{n-1}$st power in $K_n\left(\sqrt[3]{3^{2n+1}\binom{b}{3}}\right)$.  
Thus $$K_n(d, \sqrt[3^{n-1}]{d}) = K_n\left(\sqrt[3]{3^{2n+1}\binom{b}{3}}, \sqrt[3^{n-1}]{\frac{a}{a+b}}\right),$$ finishing the proof of (iii).
\\
\\
\emph{To (iv):}
Here $d = \frac{a}{a+b}$, and our model of $f$ is defined over $K_n$.
By Corollary \ref{Cnoinseptails} (iii), there is an inseparable tail $\ol{W}'$ containing the specialization of $x=1$ and a unique new inseparable tail
containing the specialization of $x=d'$.  As in parts (ii) and (iii), the fiber of $\ol{f}$ above the specialization of $x=1$ is pointwise fixed by
the absolute Galois group of $K_n(\sqrt[3^{n-s}]{\frac{a}{a+b}})$.  Likewise, the fiber of $\ol{f}$ above the specialization of $x=d'$ is fixed by
the absolute Galois group of $K_n(\sqrt[3^{n-s+1}]{c(d')^a(d'-1)^b})$.  By Proposition \ref{Pstabletail}, the stable model of $f$ is defined over a
tame extension of the compositum of these two fields, which is exactly the field given in part (iv) of the proposition.
\\
\\
\emph{To (v):}
In this case, $d = d_0$.  Note that $n \geq 2$, as there are no three-point $\ints/2$-covers.
One sees that $c = d_0^{-a}(d_0-1)^{-b} \in K_n$ (in fact, $c  \in K_3$ always, and $c \in K_2$ for $n=2$). So our model of $f$ is defined over $K_n$.

By Corollary \ref{Cnoinseptails} (iv) (and Remark \ref{Rallinsep}), there is a
unique inseparable $p^j$-tail $\ol{W}_j$ of $\ol{X}$ for each $1 \leq j < s$.  Also, there is an inseparable tail containing the
specialization of $x=1$ (even if these inseparable tails did not exist, our proof would still carry through---only our 
$K$ would overestimate the minimal field of definition of 
the stable model).  Each tail $\ol{W}_j$ contains the specialization of $x=d_j$ to a smooth point of $\ol{X}$.

As in (iv), the fiber of $\ol{f}$ above the specialization of $x=d_j$, for $1 \leq j < s$,
is pointwise fixed by $G_{L_j}$, where $L_j = K_n\left(\sqrt[2^{n-j}]{\frac{d_j^a(d_j-1)^b}{d_0^a(d_0-1)^b}}\right)$.  
As in (ii) and (iii), 
the fiber above the specialization of $x=1$ is pointwise fixed by $G_{L'}$, where $L' = K_n\left(\sqrt[2^{n-s}]{d_0^{-a}(d_0-1)^{-b}}\right)$.  
Keeping in mind that $v(b) = n-s$, we see that $K$ 
(as defined in the proposition) contains the compositum of $L'$ and all the extensions $L_j$.  We conclude using Proposition 
\ref{Pstabletail}.
\end{proof}

\begin{corollary}\label{Ccyclicstable}
In each case covered in Proposition \ref{Pm1stable}, the $n$th higher ramification group of $K^{stab}/K_0$ for the upper numbering vanishes.
\end{corollary}

\begin{proof} We first note that any tame extension of a Galois extension of $K_0$ is itself Galois over $K_0$.
In case (i) of Proposition \ref{Pm1stable}, $K^{stab}$ is contained in a
tame extension of $K_n$.  The $n$th higher ramification groups for the upper
numbering for $K_n/K_0$ vanish by \cite[Corollary to IV, Proposition 18]{Se:lf}.  By Lemma \ref{Ltamenochange}, the $n$th higher 
ramification groups vanish for $K^{stab}/K_0$ as well.  

For case (ii) of Proposition \ref{Pm1stable}, we note that $v(\frac{a}{a+b}) = 0$.
So $K_n(\sqrt[p^{n-s}]{\frac{a}{a+b}})/K_0$ has trivial $n$th higher
ramification groups for the upper numbering by \cite[Theorem 5.8]{Vi:ra}.  We again conclude using Lemma 
\ref{Ltamenochange}.

For cases (iii) and (iv) of Proposition \ref{Pm1stable}, Lemma \ref{Lsmallprimeram} shows that $K^{stab}$ is a tame extension of an extension of 
$K_0$ for which the $n$th higher ramification groups for the upper numbering vanish.  For case (v), this fact is shown by Proposition \ref{Pappendixmain}.
We again conclude using Lemma \ref{Ltamenochange}.
\end{proof}

\subsection{The general $p$-solvable case}\label{Spsolvable}

We maintain the notation of \S\ref{Smain}.

\begin{prop}\label{Pmainpsolvable}
Let $G$ be a $p$-solvable finite group with a cyclic $p$-Sylow subgroup $P$ of order $p^n$. 
If $f:Y \to X$ is a three-point $G$-cover of $\proj^1$ defined over $\ol{K_0}$,
then there exists a field extension $K'/K_0$ such that \vspace{0.05cm}
\begin{enumerate}[(i)]
\item The cover $f$ has a model whose stable model is defined over $K'$.
\item The $n$th higher ramification group of $K'/K_0$ for the upper numbering vanishes. 
\end{enumerate}
In particular, if $K$ is the field of moduli of $f$ relative to $K_0$, then $K \subseteq K'$, so the $n$th higher ramification group of $K/K_0$ for the upper 
numbering vanishes.
\end{prop}

\begin{proof}
By Proposition \ref{Ppsolvable}, we know that there is a prime-to-$p$ subgroup $N$ 
such that $G/N$ is of the form $\ints/p^n \rtimes \ints/m_G$.
Let $f^{\dagger}: Y^{\dagger} \to X$ be the quotient $G/N$-cover.  

Suppose first that $f^{\dagger}$ is a three-point cover. Then
we know from Propositions \ref{Psolvstabtame} and \ref{Pm1stable}, along with Corollary \ref{Ccyclicstable}, that there exists a model of $f^{\dagger}$ 
whose stable model can be defined over a field $K^{stab}$ such that the $n$th higher ramification groups for the upper numbering for $K^{stab}/K_0$ vanish.
Let $\ol{f}^{\dagger}: \ol{Y}^{\dagger} \to \ol{X}^{\dagger}$ be the stable reduction of $f^{\dagger}$.
The branch points of $Y \to Y^{\dagger}$ are all ramification points of
$f^{\dagger}$, because $f^{\dagger}$ is branched at three points.
Thus, by definition, their specializations do not coalesce on $\ol{Y}^{\dagger}$.  Since $G_{K^{stab}}$ acts trivially on $\ol{Y}^{\dagger}$, it
permutes the ramification points of $f^{\dagger}$ trivially, and thus these points are defined over $K^{stab}$.  By Proposition
\ref{Pstr2aux}, the stable model $f^{st}$ of $f$ can be defined over a tame extension $K'/K^{stab}$.
By Lemma \ref{Ltamenochange}, $K'$ satisfies the properties of the proposition.

Now, suppose that $f^{\dagger}$ is branched at fewer than three points.  Since $\cf(\ol{K_0}) = 0$, the cover $f^{\dagger}$ must be a $\ints/p^n$-cover branched at
two points, say (without loss of generality) $0$ and $\infty$.  Then the
branch points of $Y \to Y^{\dagger}$ include the points of $Y^{\dagger}$ lying over $x=1$, as well as the ramification
points of $f^{\dagger}$.  We may assume that $f^{\dagger}: Y^{\dagger} \to X$ is given by the equation
$y^{p^n} = x$, which is defined over $K_n$ as a $\ints/p^n$-cover.  Then, the points lying above $x=1$ are also defined 
over $K_n$.  By Proposition \ref{Pstr2aux}, we can take $K'$ to be a tame extension of $K_n$.  
The $n$th higher ramification group of $K_n/K_0$ for the upper numbering vanishes 
(\cite[Corollary to IV, Proposition 18]{Se:lf}).  By Lemma \ref{Ltamenochange}, 
the $n$th higher ramification group of $K'/K_0$ for the upper numbering vanishes.
\end{proof}

Theorem \ref{Tmain} now follows from Propositions \ref{Plocal2global} and \ref{Pmainpsolvable}.

\begin{remark}\label{Rlessexplicit}
The proofs of Propositions \ref{Psolvstabtame} and \ref{Pm1stable}, and Corollary \ref{Ccyclicstable}, which are the main ingredients in the proof of 
Theorem \ref{Tmain}, depend on writing 
down explicit extensions and calculating their higher ramification groups.  It would be interesting to find a method to place 
bounds on the conductor without writing down explicit extensions.  Such a method might be more easily generalizable to the 
non-$p$-solvable case.
\end{remark}

\appendix

\section{Explicit determination of the stable model of a three-point $\ints/p^n$-cover, $p > 2$}\label{Aexplicit}
Throughout this appendix, we assume the notations of \S\ref{Smequals1} (in particular, that $f:Y \to X$ is given by $y^{p^n} = cx^a(x-1)^b$ for some 
$c$, and $d$ is as in Lemma \ref{Lonenewtail}).  So $G \cong \ints/p^n$ and $Q_i$ $(0 \leq i \leq n)$ is the unique subgroup of order $p^i$.
For a three-point $G$-cover $f$ defined over $\ol{K_0}$, 
the methods of \S\ref{Smequals1} are sufficient to bound the conductor of the field of 
moduli of $f$ above $p$.   But we can also completely determine the structure of the stable model of $f$ (Propositions 
\ref{Pfullbranchstable}, \ref{Ppartbranchstable}, and \ref{Ptopstable}).  Although this is essentially already done in \cite[\S 3]{CM:sr}, we include this 
appendix for three reasons.  First, we compute the stable reduction of the cover $f$, as opposed to the curve $Y$.  Second, we have fewer restrictions
than Coleman in the case $p = 3$ (we allow not only covers with full ramification above all three branch points, but also covers with ramification
index $3$ above one of the branch points).  Most importantly, our proof 
requires significantly less computation and guesswork, and takes advantage of the vanishing cycles formula, as well as the 
effective different (Definition \ref{Deffdifferent}). Indeed, the majority of the computation required is already encapsulated in Lemma \ref{Lonenewtail}.

While it would be a somewhat tedious calculation, our proof can be adapted to the case of all cyclic three-point covers without using new techniques.  
However, for simplicity, we 
assume throughout this appendix that either $p > 3$, or that $p=3$ and either $f$ is totally ramified above three points, or $f$ is totally ramified 
above two points and ramified of index $3$ above the third. 

\begin{lemma}\label{Lno2jumps}
The stable reduction $\ol{X}$ cannot have a $p^i$-component intersecting a $p^{i+j}$-component, for $j \geq 2$.
\end{lemma}

\begin{proof}
Let $\ol{X}_c$ be such a $p^i$-component.  Then, a calculation
with the Hurwitz formula shows that the genus of any component $\ol{Y}_c$ above $\ol{X}_c$ is greater than 0.
By Lemma \ref{Lgenusincluded}, $\ol{X}_c$ is a component of the stable reduction of the cover $f': Y/Q_i \to X$.
It is \'{e}tale, and thus a tail by Lemma \ref{Letaletail}.  Let $\sigma_c$ be its effective ramification invariant.  By \cite[Lemma
4.2]{Ob:vc}, $\sigma_c \geq p > 2$.  
But this contradicts the vanishing cycles formula (\ref{Evancycles}).
\end{proof}

\begin{lemma}\label{Ltwocontract}
Suppose $\ol{W}$ is a $p^i$-component of $\ol{X}$ that does not contain the specialization of a branch point of $f$,
and does not intersect a $p^j$-component with $j > i$.  Then $\ol{W}$ intersects at least three other components.
\end{lemma}
 
\begin{proof}  
Let $\ol{V}$ be an irreducible component of $\ol{Y}$ lying above $\ol{W}$.  Let $\ol{V}'$ be the smooth, proper curve with 
function field $k(\ol{V})^{p^i}$.  Then $f^{st}$ induces a natural map $\alpha: \ol{V}' \to \ol{W}$.  
By Proposition \ref{Pspecialram} (i) and (ii), this map is tamely ramified, and is branched only at points where $\ol{W}$ intersects another 
component.  If there are only two such points, then $\alpha$ is totally ramified, and $\ol{V}'$ (and thus $\ol{V}$) has genus 
zero.  This violates the three-point condition of the stable model.
\end{proof}

We now give the structure of the stable reduction when $f$ has three totally ramified points.
\begin{prop}\label{Pfullbranchstable}
Suppose that $f$ is totally ramified above all three branch points.  
Then $\ol{X}$ is a chain, with one $p^{n-i}$-component $\ol{X}_i$ for each $i$, $0 \leq i \leq n$ ($\ol{X}_0$ is the
original component).  For each $i > 0$, the component $\ol{X}_{n-i}$ corresponds to the closed disk of radius $p^{-\frac{1}{2}(i + \frac{1}{p-1})}$
centered at $d = \frac{a}{a+b}$.
\end{prop}

\begin{proof}
We know from Lemma \ref{Lonenewtail} and Corollary \ref{Cnoinseptails} that $\ol{X}$ has only one tail, so it must be a
chain.  The original component contains the specializations of the branch points, so it must be a $p^n$-component.  By,
Lemma \ref{Lno2jumps}, there must be a $p^{n-i}$-component for each $i$, $0 \leq i \leq n$.  Also, by Lemma \ref{Ltwocontract}, 
there cannot be two intersecting components $\ol{W} \prec \ol{W}'$ of $\ol{X}$ with the same size inertia groups.  
Since $\ol{f}$ is monotonic (Proposition 
\ref{Pmonotonic}), there must be exactly one $p^i$-component of $\ol{X}$ for each $0 \leq i \leq n$.

It remains to show that the disks are as claimed.  For $i = n$, this follows from Lemma \ref{Lonenewtail}.  For $i < n$,
consider the cover $Y/Q_{n-i} \to X$.  
The stable model of this cover
is a contraction of $Y^{st}/Q_{n-i} \to X^{st}$.  By Lemma \ref{Lonenewtail} (using $i$ in place of $n$), the stable
reduction of $Y/Q_{n-i} \to X$ has a new \'{e}tale tail corresponding to a closed disk centered at $d$ with radius 
$p^{-\frac{1}{2}(i+ \frac{1}{p-1})}$.  Thus $\ol{X}$ also contains such a component.  This is true for every $i$,
proving the proposition.
\end{proof}

Things are more complicated when $f$ has only two totally ramified points:
\begin{prop}\label{Ppartbranchstable}
Suppose that $f$ is totally ramified above $0$ and $\infty$, and ramified of index $p^s$ above $1$, for some $0 < s < n$.
If $p=3$, assume further that $s=1$.
Then the augmented dual graph $\mc{G}'$ of the stable reduction of $\ol{X}$ is as in Figure \ref{Fstable}.
\begin{figure}[htp]
\centering
\includegraphics[totalheight=.15\textheight, width=.9\textwidth]{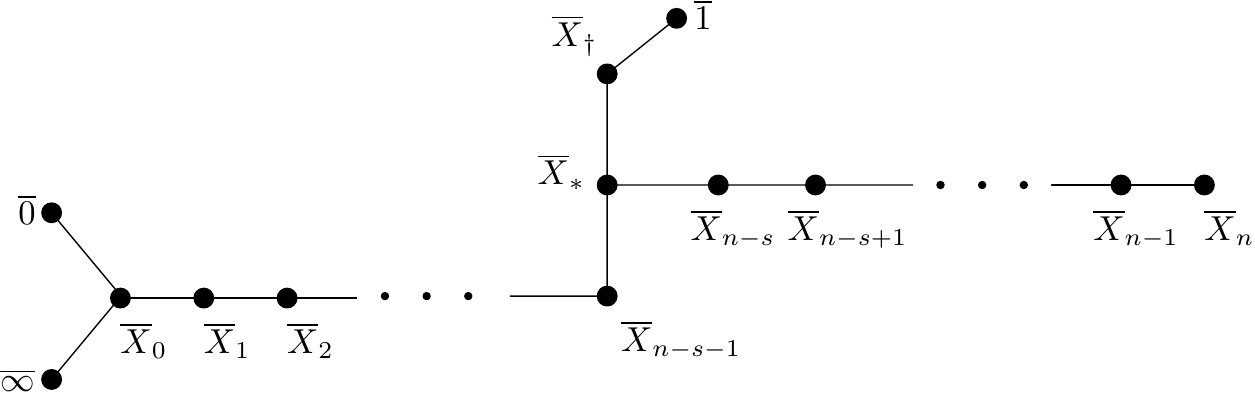}
\caption{The augmented dual graph $\mc{G}'$ of the stable reduction of a three-point $\ints/p^n$-cover with two totally ramified points, $p \ne 2$}\label{Fstable}
\end{figure}

In particular, as labeled in Figure \ref{Fstable}, the original component $\ol{X}_0$ is a $p^n$-component.  
For $s+1 \leq i < n$, $\ol{X}_{n-i}$ is a $p^{i}$-component corresponding to the disk of radius $p^{-\frac{1}{2}(n-i+\frac{1}{p-1})}$ centered at $d$.
For $0 \leq i \leq s$, $\ol{X}_{n-i}$ is a $p^{i}$-component corresponding to the disk of radius $p^{-\frac{1}{2}(2n-s-i+\frac{1}{p-1})}$ centered at $d$.
The component $\ol{X}_*$ is a $p^{s+1}$-component corresponding to the disk of radius $p^{-(n-s)}$ centered at $d$.
The component $\ol{X}_{\dagger}$ is a $p^s$-component corresponding to the disk of radius $p^{-(n-s+\frac{1}{p-1})}$ centered at $1$.
The vertices corresponding to $0$, $1$, and $\infty$ are marked as $\ol{0}$, $\ol{1}$, and $\ol{\infty}$.
\end{prop}

\begin{proof}
Recall that $v(1-d) = v(b) = n-s$, so long as $p \geq 3$.  By Corollary \ref{Cinsepprimitive} and Lemma 
\ref{Lonenewtail}, $\ol{X}$ contains exactly two tails: an inseparable tail $\ol{X}_{\dagger}$ containing the specialization of $x=1$; 
and an \'{e}tale tail $\ol{X}_n$ containing the specialization of $d$. 
There must be a component of $\ol{X}$
``separating" $1$ and $d$, i.e., corresponding to the disk centered at $d$ (equivalently, $1$) of radius $|1-d| = p^{-(n-s)}$.  Call
this component $\ol{X}_*$.  Then $\ol{X}$ looks like a chain from the original component $\ol{X}_0$ to $\ol{X}_*$,
followed by two chains, one going out to $\ol{X}_{\dagger}$ and one going out to $\ol{X}_n$.  

Let us first discuss the component $\ol{X}_*$.  Consider the cover $f': Y^{st}/Q_s \to X^{st}$.  For any edge $e$ of $\mc{G}'$ corresponding to a 
singular point on $\ol{X}$, we will take $(\sigma^{\eff}_e)'$ to mean the effective invariant for the cover $f'$.
Now, the generic fiber of $f'$ is a cover branched at two points, so $Y^{st}/Q_s$ has genus 0 fibers.  By Lemma \ref{Lartinschreiergenus}, any tail 
$\ol{X}_b$ for the blow-down of the special fiber $\ol{f}'$ of $f'$ to a stable curve
must have $\sigma_b = 1$.  Lemma \ref{Lsigmaeff} shows that, if $s(e) \prec t(e)$, then
$(\sigma^{\eff}_e)' = 1$.  Since the deformation data above $\ol{X}_0$ are multiplicative, the effective different
$(\delta^{\eff})'$ for $f'$ above $\ol{X}_0$ is $n - s + \frac{1}{p-1}$.  So above $\ol{X}_*$, it is $n-s + \frac{1}{p-1} - (n-s) =
\frac{1}{p-1} > 0$, by Lemma \ref{Leffdifferentepaisseur} applied to each of the singular points between $\ol{X}_0$ and
$\ol{X}_*$ in succession.   This means that $\ol{X}_*$ is an inseparable component for
$f'$, which means that it is at least a $p^{s+1}$-component for $f$.  

Next, we examine the part of $\ol{X}$ between $\ol{X}_0$ and $\ol{X}_*$.
By Lemma \ref{Lno2jumps}, there must be a $p^i$-component of $\ol{X}$ for each $i$ such that $s+1 \leq i
\leq n$.  Then if we take 
$f'_i: Y^{st}/Q_i \to X^{st}$, the effective different for $f_i'$ above $\ol{X}_0$ is $n - i + \frac{1}{p-1}$.
As in the previous paragraph, Lemma \ref{Leffdifferentepaisseur} shows that above a component 
corresponding to the closed disk of radius $p^{-(n-i+\frac{1}{p-1})}$ centered at $d$, the effective different for $f_i'$ will be
0.  This means that this component is the innermost $p^i$-component.   In Figure \ref{Fstable}, we label this component $\ol{X}_{n-i}$.
In particular, the $p^{s+1}$-component $\ol{X}_{n-s-1}$ corresponds to the closed disk 
of radius $p^{-(n-s-1 + \frac{1}{p-1})}$ around $d$.  Note that $\ol{X}_*$ corresponds to the closed disk of radius $p^{-(n-s)}$ around $d$, and thus lies outward from 
$\ol{X}_{n-s-1}$.
By monotonicity, $\ol{X}_*$ is a $p^{s+1}$-component.  By Lemma \ref{Ltwocontract}, $\ol{X}_*$ intersects $\ol{X}_{n-s-1}$, and for $s+1 < i \leq n$, there is 
exactly one $p^i$-component, namely $\ol{X}_{n-i}$. 
So the part of $\ol{X}$ between $\ol{X}_0$ and $\ol{X}_*$ is as in Figure \ref{Fstable}, and radii of the corresponding
disks are as in the proposition.

Now, let us examine the part of $\ol{X}$ between $\ol{X}_*$ and $\ol{X}_{\dagger}$.  We have seen that $\ol{X}_*$ is a
$p^{s+1}$-component, and $\ol{X}_{\dagger}$ is a $p^s$-component by Proposition \ref{Pcorrectspec}.  So, by Lemma \ref{Ltwocontract},
 this part of $\ol{X}$ consists only of these two components.  Recall that if we quotient out $Y^{st}$ by $Q_s$, the
effective different above $\ol{X}_*$ is $\frac{1}{p-1}$.  Also, recall that the effective invariant $(\sigma^{\eff}_e)'$ for $s(e), t(e)$ corresponding to $\ol{X}_*$, $\ol{X}_{\dagger}$ 
is 1.  So by Lemma
\ref{Leffdifferentepaisseur}, the \'{e}paisseur of this annulus is $\frac{1}{p-1}$, and $\ol{X}_{\dagger}$ corresponds to the disk of radius $p^{-(n-s+\frac{1}{p-1})}$ centered
at 1.  

Lastly, let us examine the part of $\ol{X}$ between $\ol{X}_*$ and the new tail $\ol{X}_n$.  We know there must be a
$p^i$-component for each $i$, $0 \leq i \leq s+1$.  This component must be unique, by Lemma \ref{Ltwocontract}.  
These components are labeled
$\ol{X}_{n-i}$ in Figure \ref{Fstable} (with the exception of $\ol{X}_*$, which corresponds to $i =s+1$).  
We calculate the radius of the closed disk
corresponding to each $\ol{X}_{n-i}$.  For $i = s$, the radius is $p^{-(n-s+\frac{1}{p-1})}$ for the same
reasons as for $\ol{X}_{\dagger}$.
For $i=0$, we already know from Lemma \ref{Lonenewtail} that the radius is
$p^{-\frac{1}{2}(2n-s+\frac{1}{p-1})}$.
For $1 \leq i \leq s-1$, we consider the cover $Y/Q_i \to X$.  The stable model of this cover
is a contraction of $Y^{st}/Q_i \to X^{st}$.  Since $Y/Q_i \to X$ is still a three-point cover, we can use Lemma 
\ref{Lonenewtail} (with $n-i$ and $s-i$ in place of $n$ and $s$) to obtain that the stable
reduction of $Y/Q_i \to X$ has a new tail corresponding to a closed disk centered at $d$ with radius 
$p^{-\frac{1}{2}(2n-s-i+ \frac{1}{p-1})}$.  This is the component $\ol{X}_{n-i}$.  
\end{proof}

Propositions \ref{Pfullbranchstable} and \ref{Ppartbranchstable} give us the entire structure of the stable reduction $\ol{X}$.  The following proposition
gives us the structure of $\ol{Y}$.

\begin{prop}\label{Ptopstable}
Suppose we are in the situation of either Proposition \ref{Pfullbranchstable} or \ref{Ppartbranchstable}.  If $\ol{W}$ is a $p^i$-component of $\ol{X}$ 
which does not intersect a $p^{i+1}$-component, then $\ol{f}^{-1}(\ol{W})$ consists of $p^{n-i}$ connected components, each of which is  
a genus zero radicial extension of $\ol{W}$.  If $\ol{W}$ borders a $p^{i+1}$ component $\ol{W}'$, then $\ol{f}^{-1}(\ol{W})$ consists of $p^{n-i-1}$ 
connected components, each a radicial extension of an Artin-Schreier cover of $\ol{W}$, branched of order $p$ at the point of intersection $w$
of $\ol{W}$ and $\ol{W}'$.  The conductor of this cover at its unique ramification point is $2$, unless we are in the situation of Proposition \ref
{Ppartbranchstable} and $i \geq s$, in which case the conductor is $1$.   

The rest of the structure of $\ol{Y}$ is determined by the fact that $\ol{Y}$ is tree-like (i.e., the dual graph of its irreducible components is a tree).
\end{prop}

\begin{proof}
That $\ol{Y}$ is tree-like follows from \cite[Th\'{e}or\`{e}me 1]{Ra:pg}.  This means that any two irreducible components of $\ol{Y}$ can intersect
at at most one point.
Everything else except the statement about the conductors follows from Proposition \ref{Pspecialram}, 
Lemma \ref{Lno2jumps}, and the fact that if 
$H$ is a cyclic $p$-group, then an $H$-Galois cover of $\proj^1$ branched at one point with inertia groups $\ints/p$ must, in fact, be a $\ints/p$-cover. 
We omit the details.  

For the remainder of the proof, let $\ol{W}$ be a $p^i$-component intersecting a $p^{i+1}$-component $\ol{W}'$.  

Suppose we are in the situation of Proposition \ref{Ppartbranchstable} 
and $i \geq s$.  Then $Y/Q_i \to X$ is branched at two points, so $Y$ has genus 
zero.  So any component of the special fiber of $Y^{st}/Q_i$ must also have genus zero.  
Since $Q_i$ acts trivially above $\ol{W}$, every component above $\ol{W}$ must 
have genus zero.  If such a component is a radicial extension of an Artin-Schreier cover, then Lemma \ref{Lartinschreiergenus} shows that the Artin-
Schreier cover must have conductor 1.

Now, suppose that $f$ has three totally ramified points or that we are in the situation of Proposition \ref{Ppartbranchstable} and $i < s$.  Then $\ol{W}$ 
is the unique $p^i$-component of $\ol{X}$ 
(Propositions \ref{Pfullbranchstable} and \ref{Ppartbranchstable}), and is thus the unique \'{e}tale tail of the stable reduction of 
the three-point cover $f': Y/Q_i \to X$.  By Lemma \ref{Lonenewtail}, the irreducible components above $\ol{W}$ in the stable reduction of $f': Y/Q_i \to 
X$ are Artin-Schreier covers with conductor 2.   Since $\ol{W}$ is a $p^i$-component, the irreducible components of $\ol{Y}$ above $\ol{W}$ are radicial extensions of Artin-Schreier covers with conductor 2.
\end{proof}

\section{Composition series of groups with cyclic $p$-Sylow subgroup}\label{Agroups}
In this appendix, we prove Proposition \ref{Pgroups}, which shows that a finite, non-$p$-solvable group with cyclic $p$-Sylow 
subgroup has a unique composition factor with order divisible by $p$.
Before we prove Proposition \ref{Pgroups}, we prove a lemma.  Our proof depends on the classification of finite
simple groups.

\begin{lemma}\label{Lsimpleout}
Let $S$ be a nonabelian finite simple group with a (nontrivial) cyclic $p$-Sylow subgroup.  
Then any element $\ol{x} \in \Out(S)$ with order $p$ lifts to an automorphism $x \in \Aut(S)$ with order $p$.
\end{lemma}

\begin{proof}
All facts about finite simple groups used in this proof that are not clear from the definitions or otherwise cited
can be found in \cite{atlas}.

First, note that no non-abelian simple group has a cyclic 2-Sylow subgroup, so we assume $p \neq 2$. 
Note also that no primes other than $2$ divide the order of the outer
automorphism group of any alternating or sporadic group.  So we may assume that $S$ is of Lie type.

We first show that $p$ does not divide the order $g$ of the graph automorphism group or $d$ of the diagonal automorphism 
group of $S$.
The only simple groups $S$ of Lie type for which an odd prime divides $g$ are those of the form $O_8^+(q)$.  
In this case $3 | g$.  But $O_8^+(q)$ contains $(O_4^+(q))^2$ in block form, and the order of $O_4^+(q)$ is
$\frac{1}{(4, q^2-1)}(q^2(q^2-1)^2)$.  This is divisible by 3, so $O_8^+(q)$ contains the group $\ints/3 \times
\ints/3$, and does not have a cyclic 3-Sylow subgroup.

The simple groups $S$ of Lie type for which an odd prime $p$ divides $d$ are the following:
\begin{enumerate}
\item $PSL_n(q)$, for $p | (n, q-1)$.
\item $PSU_n(q^2)$, for $p | (n, q+1)$.
\item $E_6(q)$, for $p = 3$ and $3 | (q-1)$.
\item $^2E_6(q^2)$, $p = 3$ and $3 | (q+1)$.
\end{enumerate}
Now, $PSL_n(q)$ contains a split maximal torus $((\ints/q)^{\times})^{n-1}$.  Since $p |(q-1)$, this group contains
$(\ints/p)^{n-1}$ which is not cyclic, as $p | n$ and $p \ne 2$.  So a $p$-Sylow subgroup of $PSL_n(q)$ is not cyclic.
The diagonal matrices in $PSU_n(q^2)$ form the group $(\ints/(q+1))^{n-1}$, which also contains a non-cyclic $p$-group (as $p > 2$ and $p|(n, q+1)$).
The group $E_6(q)$ has a split maximal torus $((\ints/q)^{\times})^6$ (\cite[\S35]{Hu:la}), and thus contains a
non-cyclic $3$-group.  Lastly, $^2E_6(q^2)$ is constructed as a subgroup of $E_6(q^2)$.  When $q \equiv -1 \pmod{3}$,
the ratio $|E_6(q^2)|/|^2E_6(q^2)|$ is not divisible by 3, so a $3$-Sylow subgroup of $^2E_6(q^2)$ is isomorphic to one of
$E_6(q^2)$, which we already know is not cyclic.  

So if there exists an element $\ol{x} \in \Out(S)$ of order $p$, then $p$ divides $f$, the
order of the group of field automorphisms.  Also, since the group of field automorphisms is cyclic and $p$ does not 
divide $d$ or $g$, a $p$-Sylow subgroup of $\Out(S)$ is cyclic.  This means that all elements of order $p$ in $\Out(S)$ 
are conjugate in $\Out(S)$, up to a prime-to-$p$th power.  
Now, there exists an automorphism $\alpha$ in $\Aut(S)$ which has order $p$ and is not inner.
Namely, we view $S$ as the $\FF_{q}$-points of some $\ints$-scheme, where $q = \wp^f$ for some prime $\wp$, 
and we act on these
points by the $(f/p)$th power of the Frobenius at $\wp$.  Let $\ol{\alpha}$ be the image of $\alpha$ in $\Out(S)$.  Since there exists $c$ prime to $p$
such that $\ol{\alpha}^c$ is conjugate to $\ol{x}$ in $\Out(S)$, there exists some $x$ conjugate to $\alpha^c$ in $\Aut(S)$ such that
$\ol{x}$ is the image of $x$ in $\Out(S)$.  Since $\alpha^c$ has order $p$, so does $x$.  It is the automorphism we seek.
\end{proof}

The main theorem we wish to prove in this section states that a finite group with a cyclic $p$-Sylow subgroup
is either $p$-solvable or ``as far from $p$-solvable as possible."
\begin{prop} \label{Pgroups}
Let $G$ be a finite group with a cyclic $p$-Sylow subgroup $P$ of order $p^n$.  
Then at least one of the following two statements is true:
\begin{itemize}
\item $G$ is $p$-solvable.
\item $G$ has a simple composition factor $S$ with $p^n \, | \, |S|$.
\end{itemize}
\end{prop}

\begin{proof}
We may replace $G$ by $G/N$, where $N$ is the maximal prime-to-$p$ normal subgroup of $G$.
So assume that any nontrivial normal subgroup of $G$ has order divisible by $p$.
Let $S$ be a minimal normal subgroup of $G$.  Then $S$ is a direct product of isomorphic simple groups \cite[8.2, 8.3]{As:fg}.
Since $G$ has cyclic $p$-Sylow subgroup, and no nontrivial normal subgroups of
prime-to-$p$ order, we see that $S$ is a simple group with $p \, | \, |S|$.
If $S \cong \ints/p$, then \cite[Corollary 2.4 (i)]{Ob:vc} shows that $G$ is $p$-solvable.
So assume, for a contradiction, that $p^n \nmid |S|$ and $S \ncong \ints/p$.  
Then $G/S$ contains a subgroup of order $p$.  Let $H$ be the inverse image of this subgroup in $G$.  
It follows that $H$ is an extension of the form
\begin{equation}\label{Egroupextension}
1 \to S \to H \to H/S \cong \ints/p \to 1.
\end{equation}
We claim that $H$ cannot have a cyclic $p$-Sylow subgroup, thus obtaining the
desired contradiction.

To prove our claim, we show that $H$ is in fact a semidirect product $S \rtimes
H/S$, i.e., we can lift $H/S$ to
a subgroup of $H$.  Let $\ol{x}$ be a generator of $H/S$.  
We need to find a lift $x$ of $\ol{x}$ which has order $p$. 
It suffices to find $x$ lifting $\ol{x}$ such that conjugation by $x^p$ on $S$
is the trivial isomorphism, as $S$ is center-free.
Since the possible choices of $x$ correspond to the possible automorphisms of $S$ which lift the outer automorphism $\phi_{\ol{x}}$ given by $\ol{x}$, 
we need only find an automorphism of $S$ of order $p$ which lifts $\phi_{\ol{x}}$.  Since $\phi_{\ol{x}}$ has order $p$, our desired automorphism is provided by Lemma \ref
{Lsimpleout}, finishing the proof.
\end{proof}

\begin{remark}
As was mentioned in the introduction, there are limited examples of simple groups with cyclic $p$-Sylow subgroups of order greater than $p$.  
For instance, there are no sporadic groups or
alternating groups.  There are some of the form $PSL_r(\ell)$, including all groups of the form $PSL_2(\ell)$ with $v_p(\ell^2-1) > 1$ and $p$, $\ell$ odd.  
There is also the Suzuki group $Sz(32)$.  All other examples are too large to be included in \cite{atlas}.
\end{remark}

\section{Computations for $p = 2, 3$}\label{Asmallprimes}
We collect some technical computations involving small primes that would have disrupted the continuity of the main text.

For the following proposition, $R$ is a mixed characteristic $(0, 2)$ complete discrete valuation ring with residue field $k$ and fraction field $K$.
for any scheme $S$ over $R$, we write $S_k$ (resp.\ $S_K$) for $S \times_R k$ (resp.\ $S \times_R K$).
\begin{prop}\label{Pdegree2n}
Assume that $R$ contains the $2^n$th roots of unity, where $n \geq 2$.  
Let $X = \Spec A$, where $A = R\{T\}$.  
Let $f: Y_K \to X_K$ be a $\mu_{2^n}$-torsor given by the equation $y^{2^n} = s$,
where $s \equiv 1 + c_1T + c_2T^2 \pmod{2^{n+1}},$ such that $v(c_2) = n$, $c_2$ is a square in $R$, and $\frac{c_1^2}{c_2} \equiv
2^{n+1}i \pmod{2^{n+2}}$, where $i$ is (either) square root of $-1$.
Then $f: Y_K \to X_K$ splits into a union of $2^{n-2}$ disjoint 
$\mu_4$-torsors.  Let $Y$ be the normalization of $X$ in the total ring of fractions of $Y_K$. Then the map $Y_k \to X_k$ is \'{e}tale, and 
is birationally equivalent to the union of $2^{n-2}$ disjoint $\ints/4$-covers of $\proj^1_k$, each branched at one point, with 
first upper jump equal to 1.
\end{prop}

\begin{proof}
Using the binomial theorem, we see that $\sqrt[2^{n-2}]{s}$ exists in $A$ and is congruent to $1 + b_1T + b_2T^2 \pmod{8}$, with $b_1 = 
\frac{c_1}{2^{n-2}}$ and $b_2 = \frac{c_2}{2^{n-2}}$.  Then $v(b_2) = 2$, $b_2$ is a square in $R$, and $\frac{b_1^2}
{b_2} \equiv 8i \pmod{16}.$  Thus, we reduce to the case $n=2$.

Let $Z_K \cong Y_K/\mu_2$.  The natural maps $r: Z_K \to X_K$ and $q: Y_K \to Z_K$ are given by the 
equations $$z^2 = g,y^2 = z,$$ respectively.  Let us write $g' = g(1 + T\sqrt{-b_2})^2$ and $z' = z(1 + T\sqrt{-b_2})$.  Then 
$r$ is also given by the equation $$(z')^2 = g'.$$  Now, $g' = 1 + 2T\sqrt{-b_2} + \epsilon$, where $\epsilon$ is a power series 
whose coefficients all have valuation greater than 2 (note that, by assumption, $v(b_1) = \frac{5}{2}$).  By 
\cite[Ch.\ 5, Proposition 1.6]{He:ht} and Lemma \ref{Lexplicitconductor}, the torsor $r$ has (nontrivial)
\'{e}tale reduction $Z_k \to X_k$, which is 
birationally equivalent to an Artin-Schreier cover with conductor 1.  By Lemma \ref{Lartinschreiergenus}, $Z_k$ has genus zero.
Let $U$ be such that $1-2U = z'$.  Then the cover $Z_k \to X_k$ is given by the equation 
$$(\ol{U})^2 - \ol{U} = \ol{T}\sqrt{(\ol{-b_2/4})},$$ where an overline represents reduction modulo $\pi$.  Then $\ol{U}$ is a 
parameter for $Z_k$, and the normalization of $A$ in $Z_K$ is $R\{U\}$.  

It remains to show that $q$ has \'{e}tale reduction.  The cover $q$ is given by the equation 
\begin{equation}\label{Edegree4}
y^2 = z = z'(1+T\sqrt{-b_2})^{-1} = (1-2U)(1 + T\sqrt{-b_2})^{-1}.
\end{equation}
By \cite[Ch.\ 5, Proposition 1.6]{He:ht}, it will suffice to show that, up to multiplication by a square in $R\{U\}$, the right-hand side 
of (\ref{Edegree4}) is congruent to $1 \pmod{4}$ in $R\{U\}$.  Equivalently, we must show that the right-hand side is congruent 
modulo $4$ to a square in $R\{U\}$.
Modulo $4$, we can rewrite the right-hand side as 
\begin{equation}\label{Edegree42}
1 - 2U - T\sqrt{-b_2}.
\end{equation}  We also have that
$$(1-2U)^2 = z^2(1+T\sqrt{-b_2})^2 = g(1+T\sqrt{-b_2})^2 \equiv 1 + T(b_1 + 2\sqrt{-b_2}) \pmod{8}.$$
Rearranging, this yields that $$\frac{-4U+4U^2}{(b_1/\sqrt{-b_2}) + 2} \equiv T\sqrt{-b_2} \pmod{4}.$$  Since $\frac{b_1^2}{b_2} 
\equiv 8i \pmod{16}$, it is clear that $\frac{b_1^2}{-b_2} \equiv 8i \pmod{16}.$  One can then show that $\frac{b_1}{\sqrt{-
b_2}} \equiv 2 + 2i \pmod{4}$.  We obtain $T\sqrt{-b_2} \equiv 2iU - 2iU^2 \pmod{4}$. 
So (\ref{Edegree42}) is congruent to
$1 - (2 + 2i)U + 2iU^2$ modulo 4.  This is $(1 - (1 + i)U)^2$, so we are done.     
\end{proof}

\begin{lemma}\label{Lsmallprimetails}
Lemma \ref{Lonenewtail} holds when $p=2$, and also when $p=3$ and $n > s=1$.
\end{lemma}
\begin{proof}
Use the notation of Lemma \ref{Lonenewtail}, and let $R/W(k)$ be a large enough finite extension.  
As in the proof of Lemma \ref{Lonenewtail}, we must show that, 
if $D$ is the formal disk with ring of functions $R\{t\}$, then the normalization $E$ of $D$ in 
the fraction field of $R\{t\}[y]/(y^{p^n} - g(d+et))$ has special fiber with irreducible components of positive genus. 
Here $g(d+et) = \sum_{\ell=0}^{\infty} c_{\ell}t^{\ell}$, and
\begin{equation}\label{Ecell2}
c_{\ell} = e^{\ell} \sum_{j=0}^{\ell} \binom{a}{\ell-j}\binom{b}{j}d^{j-\ell}(d-1)^{-j}.
\end{equation}  
In (\ref{Ecell2}), $v(a) = 0$, $v(b) = n-s$, $v(d) = 0$,
$v(d-1) = n-s$, and $v(e) = \frac{1}{2}(2n - s + \frac{1}{p-1})$.

First, assume that $p = 3$ and $n > s = 1$.  Then $d = \frac{a}{a+b} + \frac{\sqrt[3]{3^{2n+1}\binom{b}{3}}}{a+b}$.  Using (\ref{Ecell2}), one calculates
$c_0 = 1$, $$c_1 =  e\left(\frac{(a+b)d - a}{d(d-1)}\right) = e\left(\frac{\sqrt[3]{3^{2n+1}\binom{b}{3}}}{d(d-1)}\right),$$ and $v(c_2) = n + \frac{1}{2}$.  By (\ref{Eclval}),
we have $v(c_{\ell}) \geq n + \frac{1}{2}$ except when $\ell = 3$.
Furthermore, each term in (\ref{Ecell2}) for $\ell = 3$, other than $j=3$, has valuation greater than $n + \frac{1}{2}$.  So
$$c_3 \equiv e^3\binom{b}{3}(d-1)^{-3} \pmod{3^{n+\frac{1}{2} + \epsilon}},$$ for some $\epsilon > 0$. 
Thus $v(c_1) = n + \frac{5}{12} > n$ and $v(c_3) = n + \frac{1}{4} > n$.  Note also that $v(c_{\ell}) > n + \frac{1}{2}$ for $\ell \geq 4$.
Now, $\frac{c_1^3}{3^{2n+1}} = e^3\binom{b}{3}(d-1)^{-3}d^{-3}$.  Since $v(d-1) = n-s > \frac{1}{4}$, and $v(e^3\binom{b}{3}(d-1)^{-3}) = n + \frac{1}{4}$, we obtain that
$$\frac{c_1^3}{3^{2n+1}} \equiv e^3\binom{b}{3}(d-1)^{-3} \equiv c_3 \pmod{3^{n + \frac{1}{2} + \epsilon}},$$ for some $\epsilon > 0$.
We are now in the situation (ii) of Lemma \ref{Lartinschreier} (with $h=2$), and we conclude using Lemma \ref{Lartinschreiergenus}.

Next, assume $p = 2$.
First, note that $n > s$, as there are no three-point $\ints/2^n$-covers of $\proj^1$ that are totally ramified above all three branch points.
Consider (\ref{Ecell2}).
Clearly $c_0 = 1$.  We claim that $v(c_2) = n$, that $\frac{c_1^2}{c_2} \equiv 2^{n+1}i \pmod{2^{n+2}}$, and that $v(c_{\ell}) \geq n+1$ for 
$\ell \geq 3$.  We may assume that $K$ contains $\sqrt{c_2}$.  Given the claim, we can apply Proposition \ref{Pdegree2n} to see that 
the special fiber $E_k$ of $E$ is a disjoint union of $2^{n-2}$ \'{e}tale $\ints/4$-covers of the special fiber $D_k$ of $D$, each of which extends to 
a cover $\phi: E_k' \to \proj^1_k$ branched at one point with first upper jump equal to 1.    
By \cite[Lemma 19]{Pr:lg}, such a cover has conductor at least $2$.  A 
Hurwitz formula calculation shows that the each component of $E_k'$ has positive genus, proving the lemma.

Now we prove the claim.  The term in (\ref{Ecell2}) for 
$c_2$ with lowest valuation corresponds to $j = 2$, and this term has valuation $2v(e) + v(b) - 1 - 2v(d-1)$, which is equal to 
$n$.  For $c_{\ell}$, $\ell \geq 3$, we have $v(c_{\ell}) = n + 1 + \frac{\ell-2}{2}(s + 1) - v(\ell)$ by (\ref{Eclval}). 
Since $s \geq 1$, we obtain $v(c_{\ell}) \geq n + 1$ for $\ell \geq 3$.  

Lastly, we must show that $\frac{c_1^2}{c_2} \equiv 2^{n+1}i \pmod{2^{n+2}}$.  Choose $d = \frac{a}{a+b} + \frac{\sqrt{2^nbi}}{(a+b)^2}$ as in Lemma 
\ref{Lonenewtail}.  Using (\ref{Ecell2}), we compute
$$\frac{c_1^2}{c_2} = \frac{(a(d-1)+bd)^2}{\binom{a}{2}(d-1)^2 + abd(d-1) + \binom{b}{2}d^2}.$$  Then the congruence
$\frac{c_1^2}{c_2} \equiv 2^{n+1}i \pmod{2^{n+2}}$ is equivalent to 
$$\frac{2((a+b)d -a)^2}{-bd^2} \equiv 2^{n+1}i \pmod{2^{n+2}}$$ (as the other terms in the denominator become negligible).
Plugging in $d$ to $\frac{2((a+b)d -a)^2}{-bd^2}$, we obtain 
$$\frac{2^{n+1}bi}{-b(a^2 + \frac{2a}{a+b}\sqrt{2^nbi} + \frac{2^nbi}{(a+b)^2})} \equiv 2^{n+1}i \pmod{2^{n+2}}.$$ This is 
equivalent to $\frac{-1}{a^2} \equiv 1 \pmod{2},$ as the terms in the denominator, other than $-ba^2$, are 
negligible.   This is certainly true, as $a$ is odd.  This completes the proof of the claim, and thus the lemma.
\end{proof}

\begin{lemma}\label{Lsmallprimeram}
Let $p=3$, let $n > s$ be positive integers, and let $a$ and $b$ be integers with $v_3(a) = 0$ and $v_3(b) = n-s$.
Write $K_0 = \Frac(W(k))$ and, for all $i > 0$, write $K_i = K_0(\zeta_{3^i})$, where $\zeta_{3^i}$ is a primitive $3^i$th root of unity.
If $s=1$, then the $n$th higher ramification groups for the upper numbering of 
$$K_n\left(\sqrt[3]{3^{2n+1}\binom{b}{3}}, \ \ \sqrt[3^{n-1}]{\frac{a}{a+b}}\right)/K_0$$ vanish.
If $s > 1$, let $$d' = \frac{a}{a+b} + \frac{\sqrt[3]{3^{2(n-s+1) + 1}\binom{b}{3}}}{a+b}.$$ 
Then the $n$th higher ramification groups for the upper numbering of
$$K_n\left(d', \sqrt[3^{n-s}]{\frac{a}{a+b}}, \ \ \sqrt[3^{n-s+1}]{\frac{(d')^a(d'-1)^b}{a^ab^b(a+b)^{-(a+b)}}}\right)/K_0$$ vanish.  
\end{lemma}

\begin{proof}
Assume $s=1$.  Then $3^{2n+1}{\binom{b}{3}}$ has valuation
$3n-1$.  Since $n \geq 2$, the $n$th higher ramification groups for the upper numbering of $L = K_1\left(\sqrt[3]{3^{2n+1}{\binom{b}{3}}}\right)/K_0$ 
vanish by 
\cite[Theorem 6.5]{Vi:ra}.  Also, the $n$th higher ramification groups for the upper numbering of
$L=K_n\left(\sqrt[3^{n-1}]{\frac{a}{a+b}}\right)/K_0$ vanish by \cite[Theorem 5.8]{Vi:ra}.
By Lemma \ref{Lcompositum},  $K_n\left(\sqrt[3]{3^{2n+1}\binom{b}{3}}, \sqrt[3^{n-1}]{\frac{a}{a+b}}\right)/K_0$ has trivial 
$n$th higher ramification groups for the upper numbering. 

Now, assume $s > 1$. We use case (ii) of Corollary \ref{Ccyclicstable} 
and Lemma \ref{Lcompositum} to reduce to showing that the conductor of $K/K_0$ is less than $n$, where
$$K:= K_n\left(d', \sqrt[3^{n-s+1}]{\frac{(d')^a(d'-1)^b}{a^ab^b(a+b)^{-(a+b)}}}\right)/K_0.$$
Since $v(b) = n-s$ and $v(a) = 0$, one calculates that $d'' := d'/(\frac{a}{a+b})$ can be written as $1 + r$, where $v(r) = n - s + \frac{2}{3}$.
The same is true for $(d'')^a$.  By the binomial expansion, $(d'')^a$ is a $3^{n-s}$th power in $K_n(d')$.  Thus so is 
$(d'')^a\frac{(d'-1)^b}{b^b(a+b)^{-b}}$.  So we can write $K = K_n(d', \sqrt[3]{d'''})$, for some $d''' \in K_n(d')$.  Using Lemma 
\ref{Lcompositum} again, we need only show that the conductor $h$ of $K_1(d', \sqrt[3]{d'''})/K_0$ is less than $n$. Note that $n \geq 3$.  

Let $L = K_1(d')$ and $M = K_1(d', \sqrt[3]{d'''})$.  By \cite[Lemma 3.2]{Ob:cw}, the conductor
of $L/K_1$ is $3$.  Since the lower numbering is invariant under subgroups, the greatest lower jump for the higher ramification filtration of $L/K_0$
is $3$.  Thus the conductor of $L/K_0$ is $\frac{3}{2}$.  By \cite[Lemma 3.2]{Ob:cw}, the conductor of $M/L$ is $\leq 9$.  
Applying \cite[Lemma 2.1]{Ob:cw} to $K_0 \subseteq L \subseteq M$ yields that $h$ is either $\frac{3}{2}$ or satisfies
$h \leq \frac{3}{2} + \frac{1}{6}(9 - 3) < 3 \leq n$. 
\end{proof}

For the rest of the appendix, $K_0$ is the fraction field of $W(k)$, where $k$ is algebraically closed of characteristic $2$.  We set $K_r := 
K_0(\zeta_{2^r})$.

We state an easy lemma from elementary number theory without proof:
\begin{lemma}\label{Lwhichell}
Choose $d \in \rats$, and a square root $i$ of $-1$ in $K_2$.  Let $v$ be the standard $2$-adic valuation.
If $v(d)$ is even, then $di$ is a square in $K_3$, but not in $K_2$.  Also, $d$ is a square in $K_2$.  
If $v(d)$ is odd, then $di$ is a square in $K_2$, and $d$ is a square in $K_3$.
\end{lemma}

We turn to the field extension in Proposition \ref{Pm1stable}(v).  Recall that $n \geq 2$ is a 
positive integer, and $s$ is an integer satisfying $0 < s < n$.  Also, $a$ is an odd integer and $b$ is an integer exactly 
divisible by $2^{n-s}$.  For each $0 \leq j < s$, set 
$$d_j = \frac{a}{a+b} + \frac{\sqrt{2^{n-j}bi}}{(a+b)^2},$$ where $i^2 = -1$ and the square root symbol represents either square root.  
Lastly, as in Proposition \ref{Pm1stable} 
(iii), set $$K := K_n\left(\sqrt[2^{n-1}]{d_0}, \sqrt[2^{s-1}]{d_0 -1}, \sqrt[2^{n-j}]{d_j},
\sqrt[2^{s-j}]{d_j-1}\right)_{1 \leq j < s}.$$

For the purpose of the proof below, we let $v_{\ell}$ be the valuation on $K_{\ell}$ normalized so that a uniformizer has valuation $1$
(in contrast to the convention for the rest of this paper).
\begin{prop}\label{Pappendixmain}
Let $L$ be the Galois closure of $K$ over $K_0$.  Then the conductor $h_{L/K_0}$ is less than $n$.
\end{prop}

\begin{proof}
Note that $h_{K_n/K_0} = n-1$.  Thus, by
Lemma \ref{Lcompositum}, we need only consider the extensions 
$K_{n-1}(\sqrt[2^{n-1}]{d_0})$, $K_{s-1}(\sqrt[2^{s-1}]{d_0
-1})$, $K_{n-j}(\sqrt[2^{n-j}]{d_j})$ $(1 \leq j < s)$, and $K_{s-j}(\sqrt[2^{s-j}]{d_j-1})$ $(1 \leq j < s)$ of $K_0$, 
and we consider them separately.  Write $\ell(j)$ for the smallest $\ell$ such that $d_j \in K_{\ell}$.
By Lemma \ref{Lwhichell}, we have that $\ell(j) = 2$ for $s+j$ odd and $\ell(j)=3$ for $s+j$ even.  By \cite[Corollary 4.4]{Ob:cw}, 
we need only consider those fields $K_{c}(\sqrt[2^c]{d_j})$ and $K_c(\sqrt[2^c]{d_j-1})$ such that $c + \ell(j) > n$.  
Since $\ell(j) \leq 3$ for all $j$, we are reduced to bounding the conductors of the (Galois closures of the) following 
fields over $K_0$:
$$K_{n-1}(\sqrt[2^{n-1}]{d_j})_{j \in \{0, 1\}},\  K_{n-2}(\sqrt[2^{n-2}]{d_2}),\  K_{n-2}(\sqrt[2^{n-2}]{d_j - 1})_{j \in \{0,1\}}.$$  
For any of the above field extensions involving $d_j$, we may assume that $s > j$.

\begin{itemize}
\item Let $M$ be the Galois closure of $K_{n-1}(\sqrt[2^{n-1}]{d_j})$ over $K_0$, where $j \in \{0, 1\}$. 
First, assume $\ell(j) = 2$.  Then $v_2(d_j - 1) = v_2(b) = 2(n-s) > 1$.  We conclude using 
\cite[Corollary 4.4]{Ob:cw} (with $c = n-1$, $\ell = 2$, and $t_{\alpha} = d_j - 1$) that $h_{M/K_0} < n$.

Now, assume $\ell(j) = 3$.  Suppose $d_j \in (K_3^{\times})^2$.  We know that $v_3(d_j - 1) = 4(n-s) \geq 4$.
By \cite[Lemma 3.1]{Ob:cw}, if $(d_j')^2 = d_j$, then $v_3(d_j' - 1) >1$.  Then $M$ is the Galois closure of  
$K_{n-1}(\sqrt[2^{n-2}]{d_j'})$ over $K_0$.  We conclude using 
\cite[Corollary 4.4]{Ob:cw} (with $c = n-2$, $\ell = 3$, and $t_{\alpha} = d_j' - 1$) that $h_{M/K_0} < n$.

Lastly, suppose $d_j \notin (K_3^{\times})^2$.  Write $d_j = \alpha_j \beta^2$, where $\beta^2 = \frac{a}{a+b}$, and $\beta \in K_2$ 
(Lemma \ref{Lwhichell}).  Write $\alpha_j = 1 + t_j$.  
One can find $\gamma_j \in K_3$ such that $\alpha_j = \alpha_j' \gamma_j^2$, where $\alpha_j' = 1 + t_j'$ and $v_3(t_j')$ is odd 
(this is the main content of \cite[Lemma 3.2(i)]{Ob:cw}).  Then $d_j = \alpha'_j (\beta\gamma_j)^2$.  By \cite[Remark 3.4]{Ob:cw}, we have 
$$v_3(t_j') \geq v_3(t_j) = 4(n - \frac{s+j}{2}) > 5,$$ the last inequality holding because $n > s > j$.
This means that 
$$v_3(\gamma_j^2 - 1) = v_3(\frac{\alpha_j - \alpha_j'}{\alpha_j'}) \geq v_3(t_j)  > 5.$$
Also, $v_3(\beta^2 - 1) = 4(n-s) \geq 4$.  So $v_3((\beta\gamma_j)^2 - 1) \geq 4$.
By \cite[Lemma 3.1]{Ob:cw}, we obtain $v_3(\beta\gamma_j - 1) >1$.  We conclude using 
\cite[Corollary 4.3]{Ob:cw} (with $c = n-1$, $\ell = 3$, $t_{\alpha'} = t_j'$, and $t_{\beta} =  \beta\gamma_j -1$), that $h_{M/K_0} < n$.

\item Let $M$ be the Galois closure of $K_{n-2}(\sqrt[2^{n-2}]{d_2})$ over $K_0$.  If $\ell(2) = 2$, then $h_{M/K_0} < n$ by
\cite[Corollary 4.4]{Ob:cw} (with $c = n-2$ and $\ell = 2$).  So assume $\ell(2) = 3$. 
Since $v_3(d_2 - 1) = 4(n-s) > 1$, we obtain that $h_{M/K_0} < n$ by \cite[Corollary 4.4]{Ob:cw} (with $c = n-2$, $\ell = 3$, and
$t_{\alpha} = d_2 - 1$).

\item Let $M$ be the Galois closure of $K_{n-2}(\sqrt[2^{n-2}]{d_j-1})$ over $K_0$, where $j \in \{0,1\}$.  
As in the previous case, we may assume that $\ell(2) = 3$. 
By Lemma \ref{Lwhichell}, there exists $\gamma \in K_3$ such that $\gamma^2 = \frac{-b}{a+b}$.  Write $d_j-1 = \alpha_j' \gamma^2$. 
Then $M$ is contained in the compositum of the Galois closures $M'$ of $K_{n-2}(\sqrt[2^{n-2}]{\alpha_j'})$ and $M''$ of
$K_{n-2}(\sqrt[2^{n-3}]{\gamma})$ over $K_0$.  

Now, $v_3(\alpha_j' - 1) = 4(n - \frac{s+j}{2} - (n-s)) = 4(\frac{s-j}{2}) > 1$.
By \cite[Corollary 4.4]{Ob:cw} (with $c = n-2$, $\ell = 3$, and $t_{\alpha} = \alpha_j'-1$), we have $h_{M'/K_0} < n$.
Also by \cite[Corollary 4.4]{Ob:cw} (with $c=n-3$, $\ell = 3$, and $t_{\alpha} = \gamma-1$), we have $h_{M''/K_0} < n$.
By Lemma \ref{Lcompositum}, we have $h_{M/K_0} < n$.
\end{itemize}
\end{proof}


\begin{thebibliography}{ZZZZZZ }

\bibitem[Asc00]{As:fg}
Aschbacher, Michael.  ``Finite Group Theory," 2nd ed. Cambridge University Press, Cambridge, 2000

\bibitem[Bec89]{Be:rp}
Beckmann, Sybilla.  ``Ramified Primes in the Field of Moduli of Branched
Coverings of Curves," J. Algebra 125 (1989),
236--255.

\bibitem[Bel79]{Be:ga}
Belyi, G. V. ``Galois Extensions of a Maximal Cyclotomic Field," Izv. Akad. Nauk
SSSR Ser. Mat. 43 (1979), 267--276.

\bibitem[BW04]{BW:sr}
Bouw, Irene; Wewers, Stefan.  ``Stable reduction of modular curves," Modular Curves and Abelian Varieties,
1--22, Progr. Math., 224, Birkh\"{a}user, Basel, 2004.

\bibitem[ATLAS]{atlas}
Conway, John H. et al. ``Atlas of Finite Groups," Oxford University Press, Eynsham, 1985.

\bibitem[CM88]{CM:sr}
Coleman, Robert; McCallum, William.  ``Stable reduction of Fermat curves and Jacobi sum Hecke Characters," J. Reine Angew. Math., 385 (1988), 41--101.

\bibitem[CH85]{CH:hu}
Coombes, Kevin; Harbater, David. ``Hurwitz Families and Arithmetic Galois
Groups," Duke Math. J., 52 (1985), 821--839.

\bibitem[Col93]{Co:pv}
Colmez, Pierre. ``P\'{e}riodes des vari\'{e}t\'{e}s ab\'{e}liennes \`{a} multiplication complexe," Ann. of Math. (2) 138 (1993), no. 3, 625--683

%
\bibitem[DM69]{DM:ir}
Deligne, Pierre; Mumford, David. ``The Irreducibility of the Space of Curves of
Given Genus," Inst. Hautes \'{E}tudes Sci.
Publ. Math. 36 (1969), 75--109.

\bibitem[Epp73]{Ep:wr}
Epp, Helmut P. ``Eliminating wild ramification," Invent. Math., 19 (1973), no. 3, 235--249.

%
\bibitem[SGA1]{sga1}
Grothendieck, Alexandre; and Raynaud, Mich\`{e}le. ``Rev\^{e}tements \'{e}tales et groupe fondamental," Lecture Notes in Math., 224, Springer-Verlag,
Berlin-New York, 1971.

\bibitem[Hen98]{He:dc}
Henrio, Yannick. ``Disques et Couronnes Ultram\'{e}triques," Courbes
Semi-Stables et Groupe Fondamental en
G\'{e}om\'{e}trie Alg\'{e}brique, 21--32, Progr. Math., 187, Birkh\"{a}user
Verlag, Basel, 1998.

\bibitem[Hen00]{He:ht}
Henrio, Yannick.  ``Arbres de Hurwitz d'ordre $p$ des disques et des couronnes $p$-adic formels," Th\`{e}se, Universit\'{e} Bordeaux,
available at http://www.math.u-bordeaux1.fr/$\sim$matignon/preprints.html. 

\bibitem[Hum75]{Hu:la}
Humphreys, James E.  ``Linear Algebraic Groups," Springer-Verlag, New York, 1975.

\bibitem[Liu06]{Li:sr}
Liu, Qing.  ``Stable reduction of finite covers of curves," Compos. Math. 142 (2006), 101--118.

\bibitem[LM06]{LM:wm}
Lehr, Claus; Matignon, Michel.  ``Wild monodromy and automorphisms of curves," Duke Math. J. 135 (2006), no. 3, 569--586. 

\bibitem[Mat03]{Ma:va}
Matignon, Michel.  ``Vers un algorithme pour la r\'{e}duction stable des rev\^{e}tements $p$-cycliques de la droite projective sur un corps $p$-adique,"
Math. Ann. 325 (2003), no. 2, 323--354

\bibitem[Obu10]{Ob:fm2}
Obus, Andrew. ``Fields of moduli of three-point $G$-covers with cyclic $p$-Sylow, II," preprint.  Available at http://arxiv.org/abs/1001.3723v5.


\bibitem[Obu11a]{Ob:vc}
Obus, Andrew. ``Vanishing cycles and wild monodromy," to appear in Int. Math. Res. Notices. doi:10.1093/imrn/rnr018

\bibitem[Obu11b]{Ob:pf}
Obus, Andrew. ``On Colmez's product formula for periods of CM-abelian varieties," preprint.  Available at http://arxiv.org/abs/1107.0684v1.

\bibitem[Obu11c]{Ob:cw}
Obus, Andrew.  ``Conductors of wild extensions of local fields, especially in mixed characteristic $(0, 2),$" preprint.  Available at
http://arxiv.org/abs/1109.4776v1
 
\bibitem[Pri02]{Pr:fc}
Pries, Rachel J. ``Families of Wildly Ramified Covers of Curves," Amer. J. Math. 124 (2002), no. 4, 737--768.

\bibitem[Pri06]{Pr:lg}
Pries, Rachel J. ``Wildly Ramified Covers with Large Genus," J. Number Theory,
119 (2006), 194--209.

\bibitem[Ray90]{Ra:pg}
Raynaud, Michel. ``$p$-Groupes et R\'{e}duction Semi-Stable des Courbes," the
Grothendieck Festschrift, Vol. III,
179--197, Progr. Math., 88, Birkh\"{a}user Boston, Boston, MA, 1990.

\bibitem[Ray94]{Ra:ab}
Raynaud, Michel. ``Rev\^{e}tements de la Droite Affine en Caract\'{e}ristique
$p>0$ et Conjecture d'Abhyankar,'' Invent.
Math. 116 (1994), 425--462. 

\bibitem[Ray99]{Ra:sp}
Raynaud, Michel. ``Specialization des Rev\^{e}tements en Caract\'{e}ristique
$p>0$,'' Ann. Sci. \'{E}cole Norm. Sup.,
32 (1999), 87--126.

\bibitem[Sa\"{i}97]{Sa:rm}
Sa\"{i}di, Mohamed.  ``Rev\^{e}tements Mod\'{e}r\'{e}s et Groupe Fondamental de
Graphe de Groupes," Compositio Math. 107
(1997), no. 3, 319--338.

\bibitem[Sa\"{i}98a]{Sa:pr1}
Sa\"{i}di, Mohamed.  ``$p$-rank and semi-stable reduction of curves," C. R. Acad. Sci. Paris S\'{e}r. I 326 (1998), no. 1, 63--68.

\bibitem[Sa\"{i}98b]{Sa:pr2}
Sa\"{i}di, Mohamed.  ``$p$-rank and semi-stable reduction of curves," Math. Ann. 312 (1998), no. 4, 625--639.

\bibitem[Sa\"{i}07]{Sa:gc}
Sa\"{i}di, Mohamed.  ``Galois covers of degree $p$ and semi-stable reduction of curves in mixed characteristics," Publ. Res. Inst. Math. Sci. 43 (2007), no. 3,
661--684.

\bibitem[GAGA]{gaga}
Serre, Jean-Pierre. ``G\'{e}om\'{e}trie Alg\'{e}brique et G\'{e}om\'{e}trie
Analytique," Ann. Inst. Fourier, Grenoble 6
(1955--1956), 1--42.

\bibitem[Ser79]{Se:lf}
Serre, Jean-Pierre.  ``Local Fields," Springer-Verlag, New York, 1979.

\bibitem[Sti09]{St:af}
Stichtenoth, Henning. ``Algebraic function fields and codes," 2nd ed. Springer-Verlag, Berlin, 2009.

\bibitem[Viv04]{Vi:ra}
Viviani, Filippo.  ``Ramification Groups and Artin Conductors of Radical
Extensions of $\rats$," J. Th\'{e}or. Nombres
Bordeaux, 16 (2004), 779-816.

\bibitem[Wew03a]{We:mc}
Wewers, Stefan. ``Reduction and Lifting of Special Metacyclic Covers," Ann. Sci.
\'{E}cole Norm. Sup. (4) 36 (2003),
113--138.

\bibitem[Wew03b]{We:br}
Wewers, Stefan.  ``Three Point Covers with Bad Reduction," J. Amer. Math. Soc.
16 (2003), 991--1032.

\bibitem[Zas56]{Za:tg}
Zassenhaus, Hans J. ``The theory of groups," 2nd ed. Chelsea Publishing Company,
New York, 1956. 


\end{thebibliography}
\end{document}